\documentclass[12pt]{amsart}
\newtheorem{theorem}{Theorem}
\newtheorem{lemma}{Lemma}

\newtheorem{corollary}{Corollary}
\newtheorem{proposition}{Proposition}
\newtheorem{definition}{Definition}
\newtheorem{remark}{Remark}

\begin{document}
\title{A new look at Condition A}
\author{Quo-Shin Chi}
\thanks{The author was partially supported by NSF Grant No. DMS-0604326}
\address{Department of Mathematics, Washington University, St. Louis, MO 63130}
\email{chi@math.wustl.edu}
\date{}

\begin{abstract} Ozeki and Takeuchi~\cite[I]{OT} introduced the notion of
Condition A and Condition B
to construct two classes of inhomogeneous isoparametric
hypersurfaces with four principal curvatures in spheres, which were later generalized
by Ferus, Karcher and M\"{u}nzner to many more examples via
the Clifford representations;
we will refer to these examples of Ozeki and Takeuchi
and of Ferus, Karcher and M\"{u}nzner collectively as OT-FKM type throughout the paper.
Dorfmeister and Neher~\cite{DN} then employed isoparametric
triple systems~\cite{DN1}, which are algebraic in nature, to prove that
Condition A alone implies the isoparametric hypersurface is
of OT-FKM type. Their proof for the case of multiplicity pairs $\{3,4\}$
and $\{7,8\}$ rests on a fairly involved algebraic classification
result~\cite{Mc} about composition triples.

In light of the classification~\cite{CCJ} that leaves only the four
exceptional multiplicity
pairs $\{4,5\},\{3,4\},\{7,8\}$ and $\{6,9\}$ unsettled, it appears that
Condition A may hold the key to the classification when the multiplicity
pairs are $\{3,4\}$ and $\{7,8\}$. Thus Condition A deserves to be scrutinized
and understood more thoroughly from different angles.

In this paper, we give a fairly short and rather straightforward proof of
the result of Dorfmeister and Neher, with emphasis on the multiplicity pairs $\{3,4\}$
and $\{7,8\}$,
based on more geometric considerations. We make it explicit and apparent that the octonian algebra governs
the underlying isoparametric structure.
\end{abstract}

\maketitle

\section{Introduction} An isoparametric hypersurface $M$ in the sphere $S^{n}$ is one whose
principal curvatures and their multiplicities are fixed. We shall not
dwell on the history and development of the beautiful isoparametric story, and
shall leave it to, e.g.,~\cite{CCJ}, and the references therein. Through M\"{u}nzner's
work~\cite{M} one knows that such a hypersurface can be characterized by a
homogeneous polynomial $F:{\mathbb R}^{n+1}\rightarrow{\mathbb R}$ of
degree $g=1,2,3,4$ or $6$,
satisfying
\begin{equation}\nonumber
|\nabla F|^2(x)=g^2|x|^{2g-2},\quad (\Delta F)(x)=(m_2-m_1)g^2|x|^{g-2}/2
\end{equation}
for two natural numbers $m_1$ and $m_2$. The interpretation of $m_1$ and $m_2$
is that if we arrange the principal curvatures $\lambda_1>\cdots>\lambda_g$
with multiplicities $m_1,\cdots,m_g$, respectively, then $m_{i}=m_{i+2}$ with
index $\mod(g)$; therefore, which one is $m_1$ or $m_2$ is only a matter of convention,
by changing $F$ to $-F$ if necessary. $F$ is called the Cartan-M\"{u}nzner
polynomial, whose restriction $f$ to $S^{n}$ has values in the interval $[-1,1]$.
$f^{-1}(c),-1<c<1,$ is a one-parameter family of isoparemetric hypersurfaces
to which $M$ belongs. 
The family degenerates to two connected
submanifolds $M_{+}:=f^{-1}(1)$ and $M_{-}:=f^{-1}(-1)$, called the focal submanifolds of $M$,
of codimension $m_1+1$ and $m_2+1$, respectively.

In the case when $g=4$, Ozeki and Takeuchi~\cite[I]{OT} introduced what they
called Conditions A and B to construct two classes of inhomogeneous
isoparametric hypersurfaces. Later on, using representations of the symmetric
Clifford algebras $C_{m_1+1}'$ (following the notation of~\cite{H}), Ferus, Karcher and M\"{u}nzner~\cite{FKM} generalized
their work to construct many more isoparametric hypersurfaces
in $S^{2(m_1+m_2)+1}$;
we will refer to these examples of Ozeki and Takeuchi
and of Ferus, Karcher and M\"{u}nzner collectively as OT-FKM type
throughout the paper. The OT-FKM hypersurfaces are of multiplicities
$\{m_1,m_2\}$, where
\begin{equation}\label{multiplicity}
m_2=k\delta(m_1)-m_1-1
\end{equation}
for some integer $k>0$, and $\delta(m_1)$ is the dimension of an irreducible module
of the skew-symmetric Clifford algebra $C_{m_1-1}$ (following the notation of~\cite{H}).
These multiplicities, with the exception
of $\{m_1,m_2\}=\{2,2\}$ or $\{4,5\}$, turn out to be
exactly the multiplicities of isoparametric hypersurfaces in spheres by the work of Stolz~\cite{S}.
We will refer to~\eqref{multiplicity} as the multiplicity formula.
The author and his collaborators recently established in~\cite{CCJ} that if $m_2\geq 2m_1-1$,
then the isoparametric hypersurface is of OT-FKM type with $m_1$ and $m_2$ given in~\eqref{multiplicity}.
This leaves open only the cases in which the multiplicities
$\{m_1,m_2\}=\{4,5\},\{3,4\},\{7,8\}$ or $\{6,9\}$ by the multiplicity
formula; we refer to them as the exceptional multiplicity pairs.

One peculiar feature of the exceptional multiplicity pairs is that they are the only pairs
for which incongruent examples of OT-FKM type admit $m_1>m_2$ in~\eqref{multiplicity}.
A deeper reason for this phenomenon manifests in~\cite{CCJ}, where it is shown
that the condition $m_2\geq 2m_1-1$ warrants that an ideal generated by certain
(complexified) components of the $2$nd fundamental form is reduced, i.e., has
no nilpotent elements, at any point of $M_{+}$. The reducedness property no longer holds,
as seen by the examples of OT-FKM type,
when it comes to the exceptional multiplicity pairs.

The aforementioned examples of Ozeki and Takeuchi are of multiplicities
$(m_1,m_2)=(3,4k),(7,8k)$ of OT-FKM type. For the construction, Ozeki and Takeuchi first imposed Condition A on the isoparametric
hypersurface. That is, they stipulated that at some point $x$ of $M_{+}$,
the shape operators $S_n$ of $M_{+}$ in all normal directions $n$ have the
same kernel. Then they imposed Condition B, which says that at the same point $x$
the components of the (cubic) $3$rd fundamental form are linearly spanned by the components
of the (quadratic) $2$nd fundamental form, with coefficients being
linear functions of the coordinates of the tangent space to $M_{+}$
at $x$.

Through the work of Ferus, Karcher and M\"{u}nzner~\cite{FKM}, one knows that Condition B
always holds for the OT-FKM type. Moreover, for the OT-FKM type, Condition A
is true at
some points on the focal submanifolds of dimension $3$ or $7$ in the case of the
exceptional multiplicity pairs $\{3,4\}$ or $\{7,8\}$.

Dorfmeister and Neher then showed~\cite{DN} that in fact Condition A alone implies
that the isoparametric hypersurface is of OT-FKM type. It seems therefore that Condition A holds the
key to the unsettled cases when the multiplicity pairs are
$\{3,4\}$ and $\{7,8\}$.
Condition A thus deserves to be scrutinized and understood more thoroughly from different
angles.

Dorfmeister and Neher's approach was via the isoparametric triple systems~\cite{DN1}, which are
algebraic in nature. The proof also relies on the fairly involved algebraic classification result~\cite{Mc}
about composition triples.

In this paper, we give a fairly short and rather straightforward proof of the result
of Dorfmeister and Neher, with emphasis on the multiplicity pairs $\{3,4\}$ and $\{7,8\}$,
based on more geometric considerations. We
make it explicit and apparent that the governing force of isoparametricity is
the octonian algebra.

In Section 2, we review the octonian algebra whose left and right multiplications
by the standard purely imaginary basis elements $e_1,\cdots,e_7$,
with $e_0$ understood to be the multiplicative identity, give rise to the two inequivalent
Clifford representations $J_a$ and $J'_a,1\leq a\leq 7,$ of $C_7$ on ${\mathbb R}^{8}$.
We also review normalized
orthogonal multiplications on ${\mathbb R}^{n+1}$, which are those bilinear
binary operations $x\circ y$ such that $|x\circ y|=|x||y|$ and $e_0\circ y=y$
for all $x,y\in{\mathbb R}^{n+1}$, where $(e_0,\cdots,e_{n})$
is the standard basis.
In ${\mathbb O}$ we characterize all the normalized orthogonal multiplications as either $x\circ y=(x(y{\overline\alpha}))\alpha$
or $x\circ y=\alpha(({\overline\alpha}y)x)$, where $\alpha$ is a
unit vector in ${\mathbb O}$ with the octonian multiplication employed on the right hand side.
In particular, restricting to ${\mathbb H}$, the associativity of the quaternions
implies $x\circ y=xy$, or $=yx$ for all $x,y\in{\mathbb H}$. At this point, we introduce
the angle $\theta$ by setting $\alpha=\cos(\theta) e_0+\sin(\theta)e$ for some
purely imaginary unit $e$.

In Section 3 we recall the expansion
formula and Condition A of Ozeki and Takeuchi, and show that at a point $x\in M_{+}$
of Condition A, the $2$nd fundamental form components can be assumed to be
$p_a(U,U)=2<e_aA,B>,1\leq a\leq 7$, associated with the standard octonian multiplication,
up to an appropriate choice of bases of the eigenspaces
of the shape operator $S$ of $M_{+}$ at $x$. Here, $U=A\oplus B\oplus C$ and
$A,B,C$
are, respectively, eigenvectors of $S$ with eigenvalues $1,-1,0$. 

Section 4 introduces two points, $x^{\#}\in M_{+}$ and
$x^{*}\in M_{-}$, related to $x\in M_{+}$ of Condition A, referred to as the mirror points
of $x$. Here,
$x^{\#}$ is also of Condition A, whose 2nd fundamental form components are given
by $p_a^{\#}(V,V)=2<e_a\circ A,B>,1\leq a\leq 7,$ for a tangent vector $V$
at $x^{\#}$
with the same eigenvector components $A$ and $B$ as above, where $\circ$ is some
normalized orthogonal multiplication on the octonian algebra.
Furthermore, the $2$nd fundamental matrices at $x^{*}$ are
appropriate combination of those at $x$ and $x^{\#}$, so that the
$2$nd fundamental form $p^{*}$ at $x^{*}$ can be succinctly expressed in terms of
$\circ$ and the octonian multiplication to read $p^{*}(W,W)=-\sqrt{2}(XZ+Y\circ Z)$,
where $W=X\oplus Y\oplus Z$ is the eigenvector decomposition
of the shape operator of a tengent vector $W$ at $x^{*}$
with eigenvalues $1,-1,0$, respectively.

In Section 5 we first present the octonian setup of the isoparametric hypersurfaces
constructed by Ferus, Karcher and M\"{u}nzner. Our expression is slightly more
general than that given in~\cite{DN2} to account for all possible normalized
orthogonal mutiplications
$\circ$ at $x^{\#}$ as indicated above. We show that, for the 
hypersurfaces constructed by Ferus, Karcher and M\"{u}nzner, we can in fact
perturb the original mirror point $x^{*}$ with arbitrary $\theta$ to one at
which $\theta=0$ or $\pi$,
i.e., at which either $a\circ b=ab$ or $a\circ b=ba$ for all
$a,b\in{\mathbb O}$, so that up to isometry there are only two such hypersurfaces.
We calculate the 3rd fundamental form at $x^{*}$ to be
${\bf q}^{*}(W,W,W)=X(Y\circ Z)-Y\circ(XZ)$
with $W=X\oplus Y\oplus Z$ the same eigenvector decomposition at $x^{*}$ as before.
We then introduce the octonian setup of the isoparametric hypersurface constructed by Ozeki
and Takeuchi. This is a hypersurface of both Conditions A and B at the point $x$
of Condition A, where the 3rd fundamental form is not linear in all variables, whereas
converting to $x^{*}$ the 3rd fundamental form ${\bf q}^{*}$
turns out to be ${\bf q}^{*}(W,W,W)=(XY-YX)Z$ (the orthogonal multiplication
$\circ$ at $x^{\#}$ coincides with the octonian multiplication in this case). The fact that $q^{*}$ is linear in
the eigenvector components $X,Y,Z$ in both Ozeki-Takeuchi and Ferus-Karcher-M\"{u}nzner
examples points to that it
will be simpler to look at the 3rd fundamental form at $x^{*}$.

Section 6 paves the way for the classification of the 3rd fundamental
form at $x^{*}$, and hence of the isoparametric hypersurface of Condition A, by verifying first that at $x^{*}$
the 3rd fundamental form
${\bf q}^{*}(W,W,W)$, for a tangent vector $W=X\oplus Y\oplus Z$ with eigenvector decomposition
as before, is
indeed only linear in $X,Y$ and $Z$; therefore, we may denote
${\bf q}^{*}$ by ${\bf q}^{*}(X,Y,Z)$ instead to treat it as a multilinear
form. We observe, by the eighth identity of the ten equations of
Ozeki and Takeuchi~\cite[I, pp 529-530]{OT} defining an isoparametric
hypersurface, that at least $|{\bf q}^{*}(X,Y,Z)|=|X(Y\circ Z)-Y\circ(XZ)|$.
We then prove several
identities of ${\bf q}^{*}(X,Y,Z)$ about what happens when one interchanges
the variables $X,Y,Z$, based
on the fifth of the ten equations of Ozeki and Takeuchi.
These properties together enable us to classify, up to an ambiguity of sign,
of the important special case ${\bf q}^{*}(X,Y,e_0)$
that the remaining classification hinges on.


In Section 7, we prove that, if $0\neq 0$ and $\pi$, then the aforementioned ambiguity
of sign can be removed and the isoparametric hypersurface
must be of the type constructed by Ferus, Karcher and M\"{u}nzner, so that
the classification is reduced to the case when $\theta=0$ or $\pi$, where
the ambiguity of sign persists to an advantage. The classification is first done
for the quaternionic case. The octonian case then follows naturally from that
the octonian algebra is two (twisted) copies of the quaternion algebra. The sign
choices then differentiate the example constructed by Ozeki and Takeuchi from
the two by Ferus, Karcher and M\"{u}nzner.

\section{The octonian algebra and Clifford representations}\label{s1} Let
${\mathbb H}$ be the quaternion algebra with the standard basis $1,i,j,k$.
The octonian
algebra ${\mathbb O}$ is ${\mathbb H}\oplus {\mathbb H}$ with the
multiplication
$$
(a,b)(c,d)=(ac-\overline{d}b,da+b{\overline c}),
$$
where overline denotes
quaternionic conjugation.
For $x=(a,b)\in {\mathbb O}$, the conjugate of $x$ is
$\overline{x}:=(\overline{a},-b)$, and the real and imaginary
parts of $x$ are $(x\pm\overline{x})/2$,
respectively. The inner product
\begin{equation}\label{eq-2}
<x,y>:=(x\overline{y}+y\overline{x})/2
\end{equation}
satisfies
\begin{eqnarray}\label{eq-1}
\aligned
<\overline{x},\overline{y}>&=<x,y>,\\
<xy,z>&=<y,\overline{x}z>=<x,z\overline{y}>,\\
x(\overline{y} z)+y(\overline{x} z)&=(zx)\overline{y}+(zy)\overline{x}=2<x,y>z.
\endaligned
\end{eqnarray}
In particular, first of all, the above formulae are the rules to follow
when we interchange
two objects in the octonian multiplication. Secondly,
when $x$ and $y$ are perpendicular and purely imaginary in
${\mathbb O}$, they satisfy
\begin{eqnarray}\label{eq0}
\aligned
xy&=-yx,\quad
x(yz)&=-y(xz),\quad
(zx)y&=-(zy)x
\endaligned
\end{eqnarray}
for all $z\in {\mathbb O}$. As a consequence of~\eqref{eq0}, if we let
$\epsilon:=(0,1)\in {\mathbb O},$
the standard orthonormal basis 
\begin{equation}\label{eq0.5}
(e_0,e_1,\cdots,e_7):=(1,i,j,k,\epsilon,i\epsilon,j\epsilon,k\epsilon)
\end{equation}
gives rise to orthogonal matrices $J_1,\cdots,J_7$ over ${\mathbb O}$, where
$J_i(z)=e_iz,1\leq i\leq 7,$ such that
$$
J_iJ_k+J_kJ_i=-2\delta_{ik}Id.
$$
Similarly, the orthogonal matrices $J_1',\cdots,J_7'$, where
$J_i'(z)=ze_i,$ satisfies
$$
J_i'J_k'+J_k'J_i'=-2\delta_{ik}Id.
$$

Recall~\cite{H} that the Clifford algebra $C_n$ (respectively, $C_n'$) is the algebra over ${\mathbb R}$
generated by
$E_1,\cdots,E_n$ subject to only the conditions that $(E_i)^2=-1$
(respectively, $(E_i)^2=1$) and $E_iE_j=-E_jE_i$ for $i\neq j$. The structure
of $C_n$ (respectively, $C_n'$, to be displayed later) is well known~\cite{H},

\begin{center}
\begin{tabular}{|c|c|c|c|c|c|c|c|c|}
\hline
n&1&2&3&4&5&6&7&8\\
\hline
$C_n$&${\mathbb C}$&${\mathbb H}$&${\mathbb H}\oplus {\mathbb H}$&${\mathbb H}(2)$&${\mathbb C}(4)$&${\mathbb R}(8)$&${\mathbb R}(8)\oplus {\mathbb R}(8)$&${\mathbb R}(16)$\\
\hline
\end{tabular}
\end{center}
subject to the periodicity condition $C_{n+8}=C_n\otimes {\mathbb R}(16)$, of which the most important ones for our purposes are $C_2={\mathbb H},
C_3={\mathbb H}\oplus {\mathbb H},C_6={\mathbb R}(8),$ the matrix
ring of size $8$-by-$8$ over ${\mathbb R}$, and
$C_7={\mathbb R}(8)\oplus {\mathbb R}(8)$. The generators $E_1,\cdots,E_n$ projected
to each irreducible summand of $C_n,n=2,3,6,7,$ give rise to $n$ matrices
$T_1,\cdots,T_n$ in ${\mathbb R}(4)$ for $C_2$ and $C_3$, and in
${\mathbb R}(8)$ for $C_6$ and $C_7$, satisfying $(T_i)^2=-Id$ and
$T_iT_j=-T_jT_i$ for $i\neq j$. These $T_i$ make ${\mathbb R}^{4}$ and ${\mathbb R}^{8}$
into irreducible $C_n$-modules. For $n=2,6$, there is only one such irreducible module as
the number of irreducible summands of $C_n$ is one, whereas for $n=3,7$, there are two inequivalent
such irreducible modules as the number of irreducible summands of $C_n$ is two. $T_1,\cdots,T_n$
are called representations of $C_n$ on the appropriate Euclidean spaces.

The upshot is that 
the octonian (respectively, quaternionic) left and right
multiplications generated above, i.e., $J_1,\cdots,J_7$ vs. $J_1',\cdots,J_7'$
(respectively, $J_1,J_2,J_3$ vs. $J_1',J_2',J_3'$)
are precisely the inequivalent representations
of $C_7$ on ${\mathbb R}^{8}$ (respectively, $C_3$ on ${\mathbb R}^4$).
These two representations are inequivalent as $J_1\cdots J_7=-Id$ whereas
$J_1'\cdots J_7'=Id$ (respectively, $J_1J_2J_3=-Id$ whereas $J_1'J_2'J_3'=Id$).

Now the subalgebra of $C_{7}$ linearly spanned by the even
products of the Clifford generators is isomorphic to
$C_6\simeq {\mathbb R}(8)$ having a single irreducible summand.
We see $J_1J_7,J_2J_7,\cdots,J_6J_7$ and 
$J'_1J'_7,J'_2J'_7,\cdots,J'_6J'_7$ are equivalent representations of $C_6$.
That is, there is an orthogonal matrix $U$ over ${\mathbb R}^{8}$ such that
$U^{-1}J_iJ_7U=J_i'J_7'$
for $1\leq i\leq 6$. A similar discussion also holds true for ${\mathbb H}$
by forgetting $e_4,\cdots,e_7$, since $C_2={\mathbb H}$. As an application, we prove
the following to be employed later.

\begin{lemma}\label{lm} Let $m=3,7$. Let $A_a,1\leq a\leq m$, be $(m+1)$-by-$(m+1)$ matrices
satisfying
\begin{equation}\label{eq'}
A_aA_b^{tr}+A_bA_a^{tr}=2\delta_{ab}Id.
\end{equation}
Then there are two orthogonal matrices $P,Q\in O(m+1)$ for which $E_a:=P^{-1}A_aQ$
satisfy $E_m=Id$, and for $1\leq a,b\leq m-1$,
$$E_aE_b+E_bE_a=-2\delta_{ab}Id.
$$
\end{lemma}

\begin{proof} Clearly we can find two orthogonal matrices $P$ and $Q$
such that $P^{-1}A_mQ=Id.$
(Take, e.g., $P=Id$ and $Q=(A_m)^{-1}$.) Set $a=m$. Then~\eqref{eq'} reduces to
\begin{eqnarray}\nonumber
E_bE_b^{tr}&=&Id,\nonumber\\
E_b+E_b^{tr}&=&0,\nonumber
\end{eqnarray}
for $1\leq b\leq m-1$. This says exactly that $E_b,1\leq b\leq m-1,$
are orthogonal matrices satisfying $(E_b)^2=-Id$ and $E_bE_c=-E_cE_b$
for $1\leq b\neq c\leq m-1$.
\end{proof}

\begin{corollary}\label{cor} Conditions and notations as in Lemma~\ref{lm}, then we may
pick orthogonal $P$ and $Q$ so that $A_a=PJ_aQ^{-1},1\leq a\leq m$. 
\end{corollary}

\begin{proof} As mentioned earlier $C_6$ is generated by
$J_1J_{m},\cdots,J_{m-1}J_{m}$. Since $C_2={\mathbb H}$ and $C_6={\mathbb R}(8)$, we know
all the Clifford representations are equivalent. Thus, there is an $O\in O(m+1)$
such that $E_a=OJ_aJ_{m}O^{-1}$
for $1\leq a\leq m-1$. Changing the $P$ and $Q$ in the above lemma to $PO$ and $QO$,
we may assume now that $E_a=J_aJ_{m},1\leq a\leq m-1$. But then changing the
(new) $P$ to $PJ_m^{-1}$, we see that we may assume $E_b=J_b$ for $1\leq b\leq m$.
\end{proof}

Recall~\cite{H} that a binary operation $\circ$ defined on ${\mathbb R}^{m+1}$ is called an
{\em orthogonal multiplication} if $|x\circ y|=|x||y|$ for all
$x,y\in {\mathbb R}^{m+1}$. Let $e_0,e_1,\cdots,e_m$ be the standard basis of
${\mathbb R}^{m+1}$. We say $\circ$ is {\em normalized} if $e_0\circ x=x$ for all
$x\in {\mathbb R}^{m+1}$; we call $({\mathbb R}^{m+1},\circ)$ a normed algebra. It is well known that if $\circ$ is normalized,
then the orthogonal maps 
$U_i(x)=e_i\circ x,1\leq i\leq m,$ satisfy $U_iU_j+U_jU_i=-2\delta_{ij}Id$
for all $1\leq i,j\leq m$. In particular, ${\mathbb R}^{m+1}$ is a $C_m$-module, which
is the case only when $m=1,3,7$. Conversely, if we have such $U_i,1\leq i\leq m$,
we let $U_0=Id$, then $e_i\circ e_j:=U_i(e_j),0\leq i,j\leq m$,
extended by linearity, gives a normalized orthogonal multiplication with
$e_0\circ x=x$ for all $x$. We identify ${\mathbb R}^{m+1}$ with ${\mathbb C},{\mathbb H}$
or ${\mathbb O},$ respectively, for $m=1,3,7$. 

\begin{lemma}\label{QS} Notation as above,
for all $z$, then there is an orthogonal transformation $T$ such that
\begin{eqnarray}
\aligned
e_a\circ T(z)&=T(e_a z)\quad{\rm or}\\
&=T(ze_a)
\endaligned
\end{eqnarray}
for $1\leq a\leq m$ and for all $z$ in the normed algebra; moreover, there is a unit
vector $\alpha$ such that $T(z)=z\alpha$ in the former case, or
$T(z)=\alpha z$ in the latter. It follows that
$$
x\circ y=(x(y\overline{\alpha}))\alpha
$$
in the former case, or
$$
x\circ y=\alpha(({\overline\alpha}y)x)
$$
in the latter. In particular,~\eqref{eq-2} and~\eqref{eq-1} remain true for $\circ$.
\end{lemma}

\begin{proof} Let $U_a(x):=e_a\circ x$. There is an orthogonal matrix $T$ such that
either $U_a=TJ_aT^{-1},$ or $U_a=TJ_a'T^{-1},1\leq a\leq m$. The first statement
follows.

To prove the second statement, we may assume $e_a\circ T(z)=T(e_a z)$ without
loss of generality. Then by the first statement just estblished, we obtain
$$
<T(u)\circ\overline{T(v)},w>=<T(u),w\circ T(v)>=<u,wv>=<u\overline{v},w>,
$$
so that
$$
T(u)\circ\overline{T(v)}=u\overline{v}.
$$
In particular, setting $\alpha:=T(e_0)$ we derive
$$
T(u)=u\circ\alpha.
$$
But then the identity $<uv,w>=<u\circ T(v),T(w)>$ implies
$$
<uv,w>=<u\circ(v\circ\alpha),w\circ\alpha>,
$$
so that when we set $v=\overline{\alpha}$ we dedece
$$
<u,w\alpha>=<u,w\circ\alpha>=<u,T(w)>
$$
for all $u,w$. That is, $T(w)=w\alpha$.

In particular, in the former case without loss of generality, we obtain
$$
x\circ y= x\circ T(T^{-1}(y))=T(xT^{-1}(y))=(x(y\overline{\alpha}))\alpha.
$$
\end{proof}

\begin{remark}\label{Q} It follows by the associativity of
${\mathbb H}$ that $x\circ y=xy$ or $=yx$ for all
$x,y\in{\mathbb H}$.
\end{remark}




Now decompose $\alpha$ as
$$
\alpha=\cos(\theta)e_0+\sin(\theta)e
$$
for some $\theta$ and some purely imaginary unit $e$.

\begin{lemma}\label{comparison} We assume $x\circ y=(x(y\overline{\alpha}))\alpha$. When orthonormal $a,b\in\text{Im}({\mathbb O})$ are such that
$(ab)e=\pm e_0$, then $a\circ b=ab$. On the other hand, when $a,b$ and $ab$ are all
perpendicular to $e$, we have
$$
a\circ b=\cos(2\theta)ab+\sin(2\theta)(ab)e
$$
\end{lemma}

\begin{proof} We assume $ab=e$ without loss of generality.
Then $b\overline{e}={\overline a}$, so that
\begin{eqnarray}\nonumber
\aligned
a\circ b&=(a(b\overline{\alpha}))\alpha\\
&=(a(\cos(\theta)b+\sin(\theta)\overline{a}))\alpha\\
&=(\cos(\theta)e+\sin(\theta)e_0)(\cos(\theta)e_0+\sin(\theta)e)\\
&=e=ab.
\endaligned
\end{eqnarray}

When $a,b,$ and $ab$ are all perpendicular to $e$, we observe first that
$$
a\circ b=(a(b\overline{\alpha}))\alpha
=-(a\overline{\alpha})(\overline{b{\overline{\alpha}}})
+2<b\overline{\alpha},\overline{\alpha}>a
=(a\overline{\alpha})(\alpha b),
$$
so that
\begin{eqnarray}\nonumber
\aligned
a\circ b&=(a(\cos(\theta)e_0-\sin(\theta)e))((\cos(\theta)e_0+\sin(\theta)e)b)\\
&=(\cos(\theta)a-\sin(\theta)ae)(\cos(\theta)b+\sin(\theta)eb)\\
&=(\cos^2(\theta)-\sin^2(\theta))ab+2\sin(\theta)\cos(\theta)(ab)e
\endaligned
\end{eqnarray}
\end{proof}



In passing, let us briefly remark that the table for $C_n'$,

\begin{center}
\begin{tabular}{|c|c|c|c|c|c|c|c|c|}
\hline
n&1&2&3&4&5&6&7&8\\
\hline
$C_n'$&${\mathbb R}\oplus {\mathbb R}$&${\mathbb R}(2)$&${\mathbb C}(2)$&${\mathbb H}(2)$&${\mathbb H}(2)\oplus {\mathbb H}(2)$&${\mathbb H}(4)$&${\mathbb C}(8)$&${\mathbb R}(16)$\\
\hline
\end{tabular}
\end{center}
subject to the periodicity condition $C_{n+8}'=C_n'\otimes {\mathbb R}(16)$, gives that the dimension of an irreducible module
of the Clifford algebra $C'_{m+1},m\geq 1$, is $2\delta(m)$, where $\delta(m)$
is the dimension of an irreducible module of $C_{m-1}$. We have
$\delta(m+8)=16\delta(m)$ and $\delta(m)=1,2,4,4,8,8,8,8$ for $m=1,\cdots,8$,
respectively.

\section{The expansion formula of Ozeki and Takeuchi} Let $M$ be an isoparametric
hypersurface with four principal curvatures in the sphere.
To fix our notation, we let $V_{+},V_{-}$ and $V_0$ be the eigenspaces of
the shape operator of $M_{+}$ in the normal direction ${\bf n}_{0}$ associated
with the eigenvalues $1,-1$ and $0$, of dimension
$m_2,m_2,m_1$, respectively. Let us agree that objects
of these eigenspaces are indexed
by $\alpha,\mu$ and $p$, respectively, so that, typical vectors (coordinates)
of $V_{+},V_{-}$ and $V_{0}$ are denoted by $e_\alpha,e_\mu,e_p$
($x_{\alpha},y_{\mu},z_{p}$), respectively, etc.

With this understood, the $2$nd fundamental matrices $S_a$ of $M_{+}$ in the
normal direction
${\bf n}_a,0\leq a\leq m_1$, upon fixing orthonormal bases $e_\alpha,e_\mu,e_p$,
are
\begin{equation}\label{2nd}
S_0=\begin{pmatrix}Id&0&0\\0&-Id&0\\0&0&0\end{pmatrix},
S_a=\begin{pmatrix}0&A_a&B_a\\A_a^{tr}&0&C_a\\B_a^{tr}&C_a^{tr}
&0\end{pmatrix},1\leq a\leq m_1,
\end{equation}
where $A_a:V_{-}\rightarrow V_{+}$,
$B_a:V_0\rightarrow V_{+}$ and $C_a:V_0\rightarrow V_{-}$.

Ozeki and Takeuchi~\cite[I, pp 523-530]{OT} obtained the
expansion formula for the Cartan-M\"{u}nzner polynomial F of $M$ as follows.
\begin{eqnarray}\label{eq0.0}
\aligned
&F(tx+y+w)=t^4+(2|y|^2-6|w|^2)t^2+8(\sum_{a=0}^{m_1}p_{a}w_{a})t\\
&+|y|^4-6|y|^2|w|^2+|w|^4-2\sum_{a=0}^{m_1}(p_{a})^2
+8\sum_{a=0}^{m_1}q_{a}w_{a}
\\
&+2\sum_{a,b=0}^{m_1}<\nabla p_{a},\nabla p_{b}>w_{a}w_{b}.
\endaligned
\end{eqnarray}
Here, $x$ is a point on $M_{+}$, $y$ is tangent to $M_{+}$ at $x$, and $w$ is
normal to $M_{+}$ with coordinates $w_i$ with respect to the chosen
orthonormal normal basis ${\bf n}_0,{\bf n}_1,\cdots,{\bf n}_{m_1}$
at $x$. Moreover, $p_a(y)$ (respectively, $q_a(y)$) is the $a$th component
of the $2$nd (respectively, $3$rd) fundamental form of $M_{+}$ at $x$. Furthermore,
$p_a$ and $q_a$ are subject to ten equations~\cite[I, pp 529-530]{OT}, of which the first three
assert that, since
$S_{\bf n}$, the $2$nd fundamental matrix of $M_{+}$ in any unit normal
direction ${\bf n}$, has eigenvalues $1,-1,0$ with fixed multiplicities, it
must be that $(S_{\bf n})^3=S_{\bf n}$. From this we can derive~\cite[II, p 45]{OT}
\begin{eqnarray}\label{eq0.6}
\aligned
&A_aA_b^{tr}+A_bA_a^{tr}+2(B_aB_b^{tr}+B_bB_a^{tr})=2\delta_{ab}Id,\\
&A_a^{tr}A_b+A_b^{tr}A_a+2(C_aC_b^{tr}+C_bC_a^{tr})=2\delta_{ab}Id,\\
&B_a^{tr}B_b+B_b^{tr}B_a=C_a^{tr}C_b+C_b^{tr}C_a,
\endaligned
\end{eqnarray}
for $a\neq b$.

A point $x\in M_{+}$ is said to be of {\em Condition A}~\cite[I]{OT} if the kernel
of $S_{\bf n}$ is $V_{0}$ for all ${\bf n}$, which amounts to the same as
saying the matrices $B_a=C_a=0$ for all $1\leq a\leq m_1$ in~\eqref{2nd}, so
that~\eqref{eq0.6} now reads
\begin{eqnarray}\label{eq0.7}
\aligned
&A_aA_a^{tr}=Id,\quad
A_aA_b^{tr}+A_bA_a^{tr}=0,\quad
A_a^{tr}A_b+A_b^{tr}A_a=0,
\endaligned
\end{eqnarray}
for $1\leq a\neq b\leq m_1$. It follows that the symmetric $2$nd fundamental matrices
$S_a,0\leq a\leq m_1$, satisfy
\begin{eqnarray}\label{eq0.8}
\aligned
&(S_a)^2=Id,\quad
S_aS_b=-S_bS_a,\forall a\neq b
\endaligned
\end{eqnarray}
when they are restricted to $V_{+}\oplus V_{-}$. In other words,~\eqref{eq0.8}
asserts that $V_{+}\oplus V_{-}\simeq {\mathbb R}^{2m_2}$ is a
$C'_{m_1+1}$-module. Hence, by the
passing remark at the end of the preceding section, we see
$m_2=k\delta(m_1)$ for some $k$; thus among $(m_1,m_2)=(2,2),(4,5),(5,4)$,
only the first is possible. (In fact, Ozeki and Takeuchi established, in their
outline~\cite[II, p 54]{OT} of the classification
of the $(2,2)$ case that had been indicated by Cartan without proof~\cite{C},
that Condition A holds on one of the focal submanifolds.) But then the multiplicity formula
$m_1+m_2+1=s\delta(m_1)$ for some $s$, with $(m_1,m_2)\neq (2,2),(4,5),(5,4),$
implies $m_1+1=(s-k)\delta(m_1)$, so that $m_1=1,3$ or $7$. In particular,
for $m_1=3$ or $7$ we always have $m_2\geq 2(m_1+1)$
when $m_2\neq m_1+1$, whereas clearly $m_2\geq 2m_1-1$ for $m_1=1$; therefore,
by the result in~\cite{CCJ} 
$M$ is of the type of multiplicity $(m_1,m_2)$ constructed by Ozeki and
Takeuchi~\cite[I]{OT} when either $m_1=1$ or $m_2\neq m_1+1$.

Thus from now on, we assume $m_2=m_1+1$ with $m_1=3,7$. Then~\eqref{eq0.7} and
Corollary~\ref{cor} give the following.

\begin{corollary}\label{cor1} At a point $x\in M_{+}$ of Condition A we may assume,
by picking appropriate bases for $V_{+}$ and $V_{-}$, that
$A_a=J_a,1\leq a\leq m_1.$
\end{corollary}

\begin{proof} The matrices $P$ and $Q$ are for the basis changes in $V_{+}$
and $V_{-}$.
\end{proof}


\section{Mirror points on $M_{+}$ and $M_{-}$}\label{s3} Assume Condition A at $x\in M_{+}$ when $(m_1,m_2)=(3,4)$ or
$(7,8)$. As above, let ${\bf n}_0,{\bf n}_1,\cdots,{\bf n}_{m_1}$ be an orthonormal
normal basis at $x$. We decompose the tangent space
to $M_{+}$ at $x$ into the eigenspaces $V_{+},V_{-},V_0$, with
coordinates $x_\alpha,y_\mu,z_p$ as aforementioned, of the shape operator
$S_{{\bf n}_0}$.
Traversing along the great circle spanned by $x$ and ${\bf n}_0$ by length
$\pi/2$,
we end up again on $M_{+}$ at ${\bf n}_0$ with $x$ as a normal vector.
Accordingly, set $x^{\#}:={\bf n}_0\in M_{+}$ and ${\bf n}^{\#}_0:=x$ normal to $M_{+}$ at
$x^{\#}$.
Then the eigenspaces $V_{+}^{\#},V_{-}^{\#},V_0^{\#}$ of
$S_{{\bf n}^{\#}_0}$ with eigenvalues $1,-1,0$ are~\cite[p 15]{CCJ},
respectively, $V_{+},V_{-},{\bf n}_0^{\perp}:={\rm span}({\bf n}_1,\cdots,{\bf n}_{m_1})$.
Moreover,
${\mathbb R}x\oplus V_0$ is the normal space to $M_{+}$ at $x^{\#}$. 

\begin{lemma}\label{mirror} $x^{\#}\in M_{+}$ is also of condition A.
\end{lemma}

\begin{proof} Although a straightforward proof can be given by
the formulae on page 15 of~\cite{CCJ}, we
choose to give one based on the expansion formula~\eqref{eq0.0}. Since
$x$ is of Condition A, we know $p_a,0\leq a\leq m_1,$ are quadratic forms in
$x_\alpha$ and
$y_{\mu}$ only. If we denote,
at $x^{\#}$, all the involved quantities in~\eqref{eq0.0} with an
additional \#, then $t^{\#}=w_0,w_0^{\#}=t,w_1^{\#}=z_1,\cdots,
w_{m_1}^{\#}=z_{m_1}$.
The $3$rd term of~\eqref{eq0.0} at $x^{\#}$, which is
$$
8(\sum_{a=0}^{m_1}p_a^{\#}w_a^{\#})t^{\#},
$$
is what determines the 2nd fundamental form at $x^{\#}$.

One obtains
$p_0^{\#}=p_0$
by the fact that $p_0w_0t=p_0w_0^{\#}t^{\#}$, which is part of
the $3$rd term of~\eqref{eq0.0} at $x$, and no other terms contribute $w_0t$
of the $1$st degree. Furthermore, expanding $8q_0w_0$ in 
$z_1,\cdots,z_{m_1}$, we have
\begin{eqnarray}\label{eq000}
\aligned
8q_0w_0&=8(H_1z_1+\cdots+H_{m_1}z_{m_1})w_0\\
&=8(H_1w_1^{\#}+\cdots+H_{m_1}w_{m_1}^{\#})t^{\#},
\endaligned
\end{eqnarray}
where $H_1,\cdots,H_{m_1}$ are quadratic forms only in
$x_{\alpha}$ and $y_{\mu}$, because $q_0$ is homogeneous of degree $1$ in all
$x_{\alpha},y_{\mu},z_{p}$~\cite[I, Lemma 15(ii), p 537]{OT}. No other terms of~\eqref{eq0.0} contribute
$z_1w_0,\cdots,z_{m_1}w_0$ of the $1$st degree. It follows that
$p_1^{\#}=H_1,\cdots,p_{m_1}^{\#}=H_{m_1}$.
Hence, $x^{\#}$ is of condition A as well.
\end{proof}
In~\eqref{2nd}, we use an additional \# to indicate the corresponding
quantities in the $2$nd fundamental matrices at $x^{\#}$.
\begin{remark}\label{rmk} Actually, Lemma~{\rm \ref{mirror}} proves more. It shows that in fact
$q_0$ determines $A_a^{\#},1\leq a\leq m_1$, whose entries are the
coefficients of $H_a/2,1\leq a\leq m_1$.
\end{remark}

Next, let
\begin{eqnarray}\nonumber
\aligned
x^{*}&=&(x+{\bf n}_0)/\sqrt{2},\quad
{\bf n}_0^{*}&=&(x-{\bf n}_0)/\sqrt{2}.
\endaligned
\end{eqnarray}
Then $x^{*}\in M_{-}$. We decompose the tangent space to $M_{-}$
at $x^{*}$ into the eigenspaces $V_{+}^{*},V_{-}^{*},V_0^{*}$,
of the shape operator $S_{{\bf n}_0^{*}}$ with eigenvalues $1,-1,0$,
respectively. Again, we use an additional * to denote all involved quantities at $x^{*}$.

\begin{lemma}\label{lemma1} We have 
\begin{description}
\item[(1)] At $x^{*}$, there holds
$V_{+}^{*}={\bf n}_0^{\perp}, 
V_{-}^{*}=V_0$,
$V_{0}^{*}=V_{-}$, and the normal space
to $M_{-}$ at $x^{*}$ is ${\mathbb R}{\bf n}_0^{*}\oplus V_{+}$
\item[(2)] The second fundamental matrices at $x^{*}\in M_{-}$
are given by the $m_1+1(=m_2)$ matrices
$$
S_a^{*}:=\begin{pmatrix} 0&0&B_a^{*}\\0&0&C_a^{*}\\(B_a^{*})^{tr}&(C_a^{*})^{tr}&0\end{pmatrix},
$$
where $1\leq a\leq m_1+1, m_1=3,7,$ and $B_a^{*}$ $(respectively,\, C_a^{*})$ is the
$m_1$-by-$(m_1+1)$ matrix formed
by stacking together, in order, the $a$th row of each of the $m_1$
matrices $-A_1/\sqrt{2},\cdots,-A_{m_1}/\sqrt{2}$ $(respectively,\,
-A_{1}^{\#}/\sqrt{2},\cdots,-A_{m_1}^{\#}/\sqrt{2})$ at $x$ $(respectively,\,
at\, x^{\#})$.
\end{description}
\end{lemma}

\begin{proof} Again we explore~\eqref{eq0.0} with a slight modification.
Namely, since~\eqref{eq0.0} is with respect to $M_{+}$ while $x^{*}\in M_{-}$,
we must consider the expansion of $-F$ at $x^{*}$ in order to apply~\eqref{eq0.0}.
From the definition of $x^{*}$
and ${\bf n}_0^{*}$, we see $t=(t^{*}+w^{*})/\sqrt{2}$ and $w_0=(t^{*}-w^{*})/\sqrt{2}$.

The collection of $(t^*)^2$ terms for $-F$ will reveal the tangent and normal space
at $x^{*}$. But these terms come from the first two terms, $8p_0w_0t$,
$-6|y|^2(w_0)^2$,$|w|^4$ and $2<\nabla p_0,\nabla p_0>w_0^2$ in the
expansion of $F$. As a result, the $2$nd term in the expansion of $-F$ at $x^{*}$ is
$$
(t^{*})^2(2(\sum_{\mu}y_{\mu}^2+\sum_{p}z_p^2+\sum_{a\geq1}w_a^2)
-6((w_0^*)^2+\sum_{\alpha}x_\alpha^2)),
$$
where as before $x_\alpha,y_\mu,z_p,w_a$ parametrize $V_{+},V_{-},V_{0}$
and the normal space to $M_{+}$ at $x$. On the other hand, the collection of
$w_0^{*}t^{*}$, which comes from the same terms,
gives $p_0^{*}$ so that we end up with
$$
p_0^{*}=\sum_{a\geq 1}w_a^2-\sum_{p}z_p^2.
$$
Hence, the first statement follows.

The collection of the terms $w_{1}^{*}t^{*}=x_{1}t^{*},\cdots,w_{m_2}^{*}t^{*}=x_{m_2}t^{*}$,
with coefficients being quadratic forms in $y_\mu,z_p,w_a,a\geq 1$,
gives rise to the $2$nd fundamental form of $M_{-}$ at $x^{*}$. But these
terms come only
from $8(\sum_{a\geq 1}p_aw_a)t^{*}/\sqrt{2}$ obtained by the third term
of~\eqref{eq0.0},
and from $8q_0t^{*}/\sqrt{2}$ obtained by the eighth term in~\eqref{eq0.0}.
Combining
them yields, by~\eqref{eq000},
$$
8\sum_{\alpha}(\sum_{a\mu}2A_{\alpha\mu a}y_\mu w_a)x_{\alpha}/\sqrt{2}
+8\sum_{\alpha}(\sum_{a\mu}2A^{\#}_{\alpha\mu a}y_\mu z_a)x_\alpha/\sqrt{2},
$$
where
$
A_a=\begin{pmatrix}A_{\alpha\mu a}\end{pmatrix},
A_a^{\#}=\begin{pmatrix}A^{\#}_{\alpha\mu a}\end{pmatrix}.
$
This is the 2nd statement, where the negative sign accounts for considering
$-F$ at $x^{*}$.
\end{proof}

Recall by Corollary~\ref{cor1} we may assume $A_{a}=J_a,1\leq a\leq m_1,$
at a point $x$ of Condition A. We now understand the structure of
$A_{a}^{\#},1\leq a\leq m_1$.
\begin{lemma}\label{veryneat} Let $e_0,e_1,\cdots,e_{m_1}$ be the standard basis of
${\mathbb R}^{m_2}\simeq {\mathbb H}$ or ${\mathbb O}$. Then $<A_a^{\#}(e_0),e_0>=0$ for all $1\leq a\leq m_1$.
In particular, we may assume $A_a^{\#}(e_0)=e_a$ for all $1\leq a\leq m_1$;
as a result, $(A_a^{\#})^{tr}(e_0)=-e_a$. It follows that we may further
assume that
$A_a^{\#}$ are skew-symmetric, i.e., that 
$A_a^{\#},1\leq a\leq m_1$, form a Clifford system.
\end{lemma}

\begin{proof} Since $A_a=J_a,1\leq a\leq m_1$, the second item in Lemma~\ref{lemma1} says that
the $a$th column of $B_a^{*}$ is zero, $1\leq a\leq m_1$.
Now, the third equation of~\eqref{eq0.6} applied to the point $x^{*}\in M_{-}$
says
\begin{equation}\label{needed}
(B_a^{*})^{tr}B_b^{*}+(B_b^{*})^{tr}B_a^{*}=
(C_a^{*})^{tr}C_b^{*}+(C_b^{*})^{tr}C_a^{*},
\end{equation}
which implies that the $a$th column of $C_a^{*}$ is also zero, $1\leq a\leq m_1$,
when we set $a=b$ in the equation.
Equivalently, this means the diagonal of $A_{a}^{\#},1\leq a\leq m_1,$ is
zero. So,
\begin{equation}\label{AS}
<A_a^{\#}(e_b),e_b>=0,1\leq a\leq m_1,0\leq b\leq m_1.
\end{equation}
Since $v_a:=A_a^{\#}(e_0),1\leq a\leq m_1,$ are perpendicular to each other
by the third equation
of~\eqref{eq0.7} and Lemma~\ref{mirror}, we deduce therefore that $v_a,1\leq a\leq m_1$,
span
$e_0^{\perp}$. Thus, there is an orthogonal matrix $\begin{pmatrix}\theta_{ab}\end{pmatrix}$
of size $m_1$-by-$m_1$ such that $\sum_{b}\theta_{ab}v_b=e_a.$ The matrices
$\sum_{b}\theta_{ab}A_b^{\#},1\leq a\leq m_1$, which are the $A$-blocks of the $2$nd
fundamental matrices corresponding to
the new normal basis
${\bf n}_0':={\bf n}_0^{\#},{\bf n}_a':=\sum_{b}\theta_{ab}{\bf n}_{b}^{\#},1\leq a\leq m_1,$
at $x^{\#}\in M_{+}$, will serve as the new $A_a^{\#}$ mapping $e_0$ to $e_a$.
Thus without loss of generality we may now assume $A_a^{\#}(e_0)=e_a,1\leq a\leq m_1$. 

In coordinates,~\eqref{needed} assumes the form
\begin{equation}\label{eq0.9}
\sum_{a=1}^{m_1}(A_{\alpha\mu a}A_{\beta\nu a}+A_{\beta\mu a}A_{\alpha\nu a})
=\sum_{b=1}^{m_1}(A^{\#}_{\alpha\mu b}A^{\#}_{\beta\nu b}+A^{\#}_{\alpha\mu b}
A^{\#}_{\beta\nu b}).
\end{equation}
Hence, if we pick $\alpha=\mu=0$ and $\beta=\nu=a,1\leq a\leq m_1$, we
see by the fact that $A_a=J_a,1\leq a\leq m_1,$ that the product of
the $(a,0)$-entry and the $(0,a)$-entry of $A_a^{\#}$ is $-1$,
so that the latter is $-1$ since the former is $1$. This forces all other entries of the first row
of $A_a^{\#}$ to be zero as $A_a^{\#}$ is orthogonal. In conclusion,
$(A_a^{\#})^{tr}(e_0)=-e_a.$ That is, $A_a^{\#}$ is skew-symmetric in the
first row and column, $1\leq a\leq m_1$.

Since $A_a^{\#},1\leq a\leq m_1,$ leave
$<e_0,e_a>^{\perp}$ invariant and since the group of automorphism of
${\mathbb H}$ and ${\mathbb O}$, which are $SO(3)$ and $G_2$, respectively,
are transitive on the unit sphere of $e_0^{\perp}$, we see that any purely imaginary
unit vector $e$ can serve as $e_1$. Therefore, $<A_a^{\#}(e),e>=0$ by~\eqref{AS}. It follows that
$A_a^{\#}$ restricted on $<e_0,e_a>^{\perp}$ is also skew-symmetric. In particular,~\eqref{eq0.7}
says that $A_a^{\#},1\leq a\leq m_1$, form a Clifford system.
\end{proof}

\begin{definition} For notational ease, we let $A_0^{\#}=Id.$ We define a
normalized orthogonal
multiplication $\circ$ on ${\mathbb R}^{m_2}$ by
$e_a\circ e_b=A_a^{\#}(e_b)$ for $0\leq a,b\leq m_1,$ and extend it by linearity.
\end{definition}


We can now determine the $2$nd fundamental form at $x^{*}\in M_{-}$.

\begin{proposition}\label{ppp2} For $(m_1,m_2)=(7,8)$, the $2$nd fundamental form ${\bf p}^{*}$ at
$x^{*}\in M_{-}$ is given by
\begin{equation}\label{2fndtl}
{\bf p}^{*}(W,W)=-\sqrt{2}(XZ+Y\circ Z)
\end{equation}
for a tangent vector $W=X\oplus Y\oplus Z$ at $x^{*}$, where
$X\in V_{+}^{*}\simeq {\rm Im}({\mathbb O}$),
the purely imaginary part of ${\mathbb O}$,
$Y\in V_{-}^{*}\simeq {\rm Im}({\mathbb O})$,
$Z\in V_{0}^{*}\simeq {\mathbb O}$, and ${\bf p}^{*}$ lives in the normal space to $M_{-}$, which is
${\mathbb R}{\bf n}_0^{*}\oplus V_{+}\simeq {\mathbb R}\oplus {\mathbb O}$.
\end{proposition}
For $(m_1,m_2)=(3,4)$, one has the same formula by forgetting the orthogonal
complement of ${\mathbb H}$ in ${\mathbb O}$.
\begin{proof} It is an immediate consequence of $<B_{a}^{*}(e_p),e_\alpha>=<e_\alpha e_p,e_a>$
and $<C_a^{*}(e_p),e_\mu>=<e_\mu\circ e_p,e_a>$.
\end{proof}
                                                                          
Henceforth, we will mainly study the structure of isoparametric
hypersurfaces in the case when $(m_1,m_2)=(7,8)$.

\section{Octonian realization of the isoparametric hypersurfaces of OT-FKM type}
\subsection{Isoparametric hypersurfaces
constructed by Ferus Karcher and M\"{u}nzner}\label{sec5}

Let ${\mathbb R}^{32}$ 
be the direct sum
of four copies of ${\mathbb O}$. We
identify
$(0,0,-e_0,0)$ with $x\in M_{+}$; $\{(0,0,Y,0):Y\in {\rm Im}({\mathbb O})\}$ with
$V_0=V_{-}^{*}$; $(0,e_0,0,0)$ with ${\bf n}_0\in M_{+}$; and
$\{(0,X,0,0):X\in {\rm Im}({\mathbb O})\}$ with $V_{+}^{*}$. We identify
$V_{-}=V_0^{*}$ with
$(Z,0,0,Z)),Z\in{\mathbb O},$ and identify $V_{+}$,
which is the normal subspace perpendicular to ${\bf n}_0^{*}$ at $x^{*}$,
with $(W,0,0,-W)$.
The notation here is in accordance with
Lemma~\ref{lemma1} and Proposition~\ref{ppp2}.

Consider the orthogonal transformations
\begin{eqnarray}\label{Cliff}
\aligned
&P_{-1}:(A,X,Y,B)\mapsto (A,-X,Y,-B)\\
&P_a:(A,X,Y,B)\mapsto (-Xe_a,-A\overline{e}_a,-B\circ {\overline e}_a,-Y\circ e_a)
\endaligned
\end{eqnarray}
for $0\leq a\leq 7$.
It is immediate that $P_iP_j+P_jP_i=2\delta_{ij}Id,-1\leq i,j\leq 7$. Therefore, the
symmetric Clifford system 
$P_{-1},P_0,\cdots,P_{7}$ over $M_{-}$ generates an isoparametric hypersurface $M$
constructed by Ferus, Karcher and M\"{u}nzner~\cite{DN2},~\cite{FKM}.

It is readily checked that
\begin{eqnarray}\label{2ndfund}
\aligned
&<P_a((Z,X,Y,Z)),(Z,X,Y,Z)>\\
&=2<XZ+Y\circ Z,e_a>,
\endaligned
\end{eqnarray}
and $<P_{-1}((Z,X,Y,Z)),(Z,X,Y,Z)>=-|X|^2+|Y|^2.$ That is, rescaling $Z$,
$-P_i,-1\leq i\leq 7,$ restricted to the tangent space to $M_{-}$ at $x^{*}$
give exactly the $2$nd fundamental form by Proposition~\ref{ppp2}. 

Recall $M_{-}$ is said to be of
{\em Condition B}~\cite[I]{OT} at $x^{*}$ if
\begin{equation}\label{CondB}
q_b^{*}=\sum_{a=-1}^{m_1}r_{ab} p_a^{*},
\end{equation}
where $r_{ab}=-r_{ba},-1\leq a,b\leq m_1;$ here, we set $q_{-1}^{*}=0$ and $p_{-1}^{*}=|X|^2-|Y|^2.$
An isoparametric hypersurface of OT-FKM type satisfies Condition B; it is
well known~\cite{FKM} that
\begin{equation}\label{B}
r_{ab}(v)=<P_a(v),n_b>,
\end{equation}
where $v$ is tangent to the focal submanifold, which is $M_{-}$ in our case, defined by the symmetric Clifford
matrices $P_a$ as the zero locus of $<P_a(x),x>=0,-1\leq a\leq 7,$
and $n_a$ are the normal basis elements.
With $n_a=(e_a,0,0,-e_a)/\sqrt{2}$ and $v=X+Y+Z$, it is straightforward to
find $r_{ab}=<e_a,Xe_b-Y\circ e_b>$ and so
\begin{equation}\label{3rdf}
{\bf q}^{*}(W,W,W)=X(Y\circ Z)-Y\circ(XZ),
\end{equation}
for a tangent vector $W=X\oplus Y\oplus Z$ at $x^{*}$, in the case of isoparametric hypersurfaces constructed by Ferus, Karcher and
M\"{u}nzner.

\subsection{Perturbing the mirror point $x^{*}$}
\begin{proposition}\label{deform} There is a point $x^{*}$ on $M_{-}$ of the
isoparametric hypersurfaces constructed by Ferus Karcher and M\"{u}nzner at which
either $a\circ b=ab$ or $a\circ b=ba$ for all $a,b\in{\mathbb O}$, up to
an isometry of the ambient Euclidean space.
\end{proposition}

\begin{proof} Similar to Lemma~\ref{QS} we can apply an orthogonal
transformation $U$ such that
$$
U(z)\circ e_a=U(ze_a)\quad\text{or}\quad U(e_az)
$$
for all $a,z$.
With $x^{\#}=(0,e_0,0,0)$ and ${\bf n}^{\#}=(0,0,n,0)$
for $n=-U(e_0)$, the normal
space to $M_{-}$ at $x^{*}_{\bf n}:=(x^{\#}+{\bf n}^{\#})/\sqrt{2}$ is spanned by 
\begin{eqnarray}\nonumber
\aligned
P_{-1}(x^{*}_{\bf n})=(0,-e_0,-U(e_0),0)/\sqrt{2},\quad\text{and}\\
P_a(x^{*}_{\bf n})=(-e_a,0,0,U(e_a))/\sqrt{2},\quad 0\leq a\leq 7,
\endaligned
\end{eqnarray}
whereas the
tangent vectors, being perpendicular to $x^{*}_{\bf n}$ and the normal vectors,
are thus of the form $(Z,X,U(Y),U(Z))$; therefore,
\begin{equation}\nonumber
\aligned
&-<P_a((Z,X,U(Y),U(Z)),(Z,X,U(Y),U(Z))>\\
&=-2<XZ+YZ,e_a>\quad\text{or}\quad -2<XZ+ZY,e_a>,
\endaligned
\end{equation}
for $0\leq a\leq 7$, give that the $2$nd fundamental
form at $x^{*}_{\bf n}$ is $-\sqrt{2}(XZ+YZ$, or $-\sqrt{2}(XZ+ZY)$
after rescaling $Z$.
\end{proof}

\subsection{Isoparametric hypersurfaces of the type constructed by Ozeki and Takeuchi}

Let ${\mathbb R}^{32}$ be identified as the direct sum
of four copies of ${\mathbb O}$. Let $x=(0,0,e_0,0)$ and at $x$ identify
$V_{+}$ as the first copy, $V_{-}$ as the second copy and the normal space as the fourth copy
of ${\mathbb O}$ in ${\mathbb R}^{32}$. Lastly, identify the imaginary
part of the third copy of ${\mathbb O}$ as $V_{0}$ at $x$. Define
\begin{eqnarray}\nonumber
\aligned
&P_0:(u,v,z,w)\mapsto (u,-v,w,z),\\
&P_a:(u,v,z,w)\mapsto (e_av,-e_au,e_aw,-e_az)
\endaligned
\end{eqnarray}
for $1\leq a\leq 7$. A calculation similar to the above one
gives that the symmetric Clifford system $P_0,P_1,\cdots,P_7$ over $M_{+}$
defines an isoparametric hypersurface M, where $x\in M_{+}$ is
of Condition A whose 2nd fundamental form is
\begin{equation}\nonumber
p_0= |u|^2-|v|^2,\quad p_a= 2<e_a, u{\overline v}>,1\leq a\leq 7.
\end{equation}
In particular, the orthogonal multiplication $\circ$ at $x^{\#}$ coincides with
the octonian multiplication. By~\cite{OT},~\cite{FKM}, we know $x$ is also of Condition B.
Indeed, with the normal basis $n_b=(0,0,0,e_b)$ and a
tangent vector $x=(u,v,z,0)$, where $u,v\in{\mathbb O}$ and
$z\in\text{Im}({\mathbb O})$, we calculate by~\eqref{CondB} to deduce
$r_{0b}=<z,e_b>,1\leq b\leq 7$ and $r_{ab}=-<e_az,e_b>, 0\leq a\neq b\leq 7$.
From this we obatin by~\eqref{B} 
\begin{eqnarray}\nonumber
\aligned
&q_0=2<z,u\overline{v}>,\\
&q_a= <z,e_a>(|u|^2-|v|^2-2<u,\overline{v}>)-2<ze_a,u\overline{v}>,
\endaligned
\end{eqnarray}
for $1\leq a\leq 7$~\cite[I, p 556]{OT}. 

Since $q_0$ gives
$A^{\#}_a,1\leq a\leq 7$, by Remark~\ref{rmk},
we see $A_a=A_a^{\#}=J_a,1\leq a\leq 7$. On the other hand,
Remark~\ref{quote}, to be given later, gives that
$$
{\bf q}^{*}=\sum_{a=0}^{m_1}w_aq_a^{*}
=<2z(u{\overline v})-2<u,v>z,w>
$$
with $w=\sum_{a=0}^{m_1}w_a e_a$. The identification
$X=w\in V_{+}^{*}\simeq {\rm Im}({\mathbb O}),Y=-z\in
V_{-}^{*}\simeq {\rm Im}({\mathbb O}),
Z=-v\in V_0^{*}$, and $W=u$ in the normal space to
$x^{*}\in M_{-}$
derives that, for a tangent vector $U=X\oplus Y\oplus Z$ and a normal vector
$W$ at $x^{*}$,
\begin{eqnarray}\nonumber
\aligned
&<{\bf q}^{*}(U,U,U),W>=\\
&<2Y(W{\overline Z})-2<W,Z>Y,X>
=<-2(YX)Z-2<X,Y>Z,W>\\
&=<2(XY)Z+2<X,Y>Z,W>
=<(XY)Z-(YX)Z,W>.
\endaligned
\end{eqnarray}
We thus arrive at
$$
{\bf q}^{*}(U,U,U)=(XY-YX)Z
$$
for a tangent vector $U=X\oplus Y\oplus Z$ at $x^{*}$. The fact that the 3rd fundamental form at $x$ of Condition A in the example of
Ozeki and Takeuchi is not linear in all
variables whereas the 3rd fundamental form is linear at $x^{*}$, in the cases
of both Ozeki-Takeuchi and Ferus-Karcher-M\"{u}nzner, in all variables points
to that it will be simpler to look at the mirror point $x^{*}$ instead.

\section{The $3$rd fundamental form at a mirror point on $M_{-}$}

Henceforth, we concentrate on $x^{*}\in M_{-}$. It is understood $(m_1,m_2)=(7,8)$.
In coordinate calculations we use $x_\alpha^{*},y_\mu^{*},z_p^{*}$ to denote
coordinates of $V_{+}^{*},V_{-}^{*},V_0^{*}$, respectively, so that
$X=\sum_{\alpha=1}^{m_1}x_\alpha^{*} e_\alpha$,
$Y=\sum_{\mu=1}^{m_1}y_\mu^{*} e_\mu$, and 
$Z=\sum_{p=0}^{m_1}z_p^{*} e_p$. We identify a normal vector perpendicular to
${\bf n}_0^{*}$ with
$W=\sum_{a=0}^{m_1}w_a e_a$.

\begin{lemma}\label{lemma2} At $x^{*}\in M_{-}$, we have $q_0^{*}=0$.
\end{lemma}

\begin{proof} This follows from Remark~\ref{rmk}. There, we see
that $q_0$ at $x\in M_{+}$ determines
$A_a^{\#},1\leq a\leq m_1$, and vice versa. Hence, if
$A_a=0,1\leq a\leq m_1,$ then~\eqref{eq0.9} derives that
$A_a^{\#}=0,1\leq a\leq m_1$,
so that $q_0=0$. Now replace $F$ by $-F$ and $x^{\#}$ by $x^{*}$ and observe
that $A_a^{*}=0,1\leq a\leq m_1$ by the second item of Lemma~\ref{lemma1}. 
\end{proof}

Now that $q_0^{*}=0$, there will be no confusion for us to change our notation from
now on to rename $q_1^{*},\cdots,q_{m_2}^{*}$, where $m_2=m_1+1$, at $x^{*}$ to be
$q_0^{*},\cdots,q_{m_1}^{*}$, so that the $3$rd fundamental form can be written as
${\bf q}^{*}=\sum_{a=0}^{m_1} q_a^{*} e_a$ in accordance with the standard octonian basis
$e_0,e_1,\cdots,e_{m_1}$.

\begin{lemma}\label{sub} At $x^{*}\in M_{-}$, the
$3$rd fundamental form ${\bf q}^{*}$ satisfies
\begin{equation}\label{sub'}
|{\bf q}^{*}(U,U,U)|=|X(Y\circ Z)-Y\circ (XZ)|
\end{equation}
for a tangent vector $U=X\oplus Y\oplus Z$ at $x^{*}$.
\end{lemma}

\begin{proof} Recall the identity for an isoparametric hypersurface~\cite[I, p 530]{OT}
\begin{equation}\label{Defining}
16|{\bf q}^{*}|^2 = 16G(|X|^2+|Y|^2+|Z|^2)-|\nabla G|^2,
\end{equation}
where $G=\sum_{a=-1}^{m_1}(p_a^{*})^2$, that an isoparametric hypersurface must
satisfy. It is understood that $p_{-1}^{*}=|X|^2-|Y|^2$.

For the isoparametric hypersurfaces of the type constructed by Ferus
Karcher and M\"{u}nzner, we know the left hand
side of~\eqref{Defining} is $|X(Y\circ Z)-Y\circ(XZ)|$ by~\eqref{3rdf} .
On the other hand, the right hand side of~\eqref{Defining} depends only on
the $2$nd fundamental form, which is exactly $-\sqrt{2}(XZ+Y\circ Z)$ for the
type constructed by Ferus, Karcher and M\"{u}nzner by~\eqref{2ndfund} and in general by Proposition~\ref{ppp2}.
\end{proof}

\begin{remark} When $m_1=1,$ the underlying normed algebra is ${\mathbb C}$.
Therefore, Lemma~{\rm \ref{sub}} implies ${\bf q}^{*}=0$. 

When $m_1=2$, Ozeki and Takeuchi established~\cite[II, p 54, Case $(B_1)$]{OT} that one can
choose appropriate coordinates so that ${\bf p}^{*}$ is identical with that of the homogeneous
example. The same argument as in Lemma~{\rm \ref{sub}} then implies
that ${\bf q}^{*}=0$ as it is so for the homogeneous example~\cite[II, p41]{OT},
so that the isoparamentric hypersurface is exactly the homogeneous one.
\end{remark}

\begin{proposition}\label{3.4-5} For $0\leq a\leq m_1$ at $x^{*}$, we have
$q_a^{*}=\sum_{\alpha\mu p} q_{a}^{\alpha\mu p} x_\alpha^{*} y_\mu^{*} z_p^{*}$ for some
coefficients
$q_{a}^{\alpha\mu p}$. That is, ${\bf q}^{*}$ is homogeneous of degree $1$
in $X,Y,Z$.
\end{proposition}

\begin{proof} We record the equation from Ozeki and Takeuchi~\cite[I, p 529]{OT},
with respect to $-F$, that
\begin{equation}\label{defining'}
<\nabla p_i^{*}, \nabla q_j^{*}>+<\nabla p_j^{*}, \nabla q_i^{*}>=0
\end{equation}
for all $-1\leq i\neq j\leq m_1$. Picking $i=-1$ and $j=a$, we get
\begin{equation}\label{eq2}
<\nabla p_{-1}^{*},\nabla q_a^{*}>=0
\end{equation}
since $q_{-1}^{*}=0$ by Lemma~\ref{lemma2}. Note that
$
p_{-1}^{*}=\sum_{\alpha}(x_\alpha^{*})^2
-\sum_{\mu}(y_{\mu}^{*})^2.
$

For the component
$
\sum_{\mu\nu p}
q_{a}^{\alpha\beta p}x_\alpha^{*} x_\beta^{*} z_p^{*}
$
of $q_a^{*}$, where $\alpha,\beta$ are in the same index range over $V_{+}^{*}$,
the left hand side of~\eqref{eq2} gives
$-4\sum_{\alpha\beta p}q_{a}^{\alpha\beta p}x_\alpha^{*} x_\beta^{*} z_p^{*}
$
(Euler's identity for
homogeneous polynomials).
Similarly for the component
$
\sum_{\alpha pq}q_{a}^{\alpha pq}x_\alpha^{*} z_p^{*} z_q^{*},
$
where $p,q$ are in the same index range over $V_0^{*}$, the left hand side of~\eqref{eq2} derives
$
2\sum_{\alpha pq} q_{a}^{\alpha pq}x_\alpha^{*} z_p^{*} z_q^{*},
$
etc. The vanishing of the right hand side of~\eqref{eq2} therefore
shows that all those components, exactly two of whose coordinates are
in the same index range, are zero. The same reasoning gives zero to the components
whose coordinates are either all in the $\alpha$-range, or all in the $\mu$-range (over $V_{-}^{*}$).
The
only component of repeated ranges not accounted for by this procedure is thus of the form
$
\sum_{pqr}q_{a}^{pqr}z_p^{*} z_q^{*} z_r^{*}
$
with $p,q,r$ in the same index range. However, Lemma 15 (i) of~\cite[I, p 537]{OT}
asserts that such components cannot exist.
\end{proof}

\begin{remark}\label{quote} ${\bf q}^{*}$ at $x^{*}\in M_{-}$ is determined
by collecting the part of $\bf q$ at $x\in M_{+}$
linear in all variables. Explicitly, since ${\bf q}^{*}$ is of degree $1$ in
$X,Y,Z$, the term $8\sum_{a=0}^{m_1} q_a^{*}w_a^{*}$ is of the form
$8\sum_{\alpha\mu p a}q^{\alpha\mu p}_{a} x_\alpha^{*} y_\mu^{*} z_p^{*}
w_a^{*}$, which is also linear in $x_\alpha,y_\mu,z_p,w_a$.
This is because by our convention, $x_\alpha,y_\mu,z_p,w_a$ parametrize,
respectively, $V_{+},V_{-},V_0$ and the normal space to
$x\in M_{+}$; we know by the first item of Lemma~{\rm \ref{lemma1}} that
$x_{\alpha}^{*}=w_\alpha,y_\mu^{*}=z_\mu, 1\leq \alpha,\mu\leq m_1$, and
$z_p^{*}=x_p,w_a^{*}=x_a,0\leq a,p\leq m_1$.
However, a glance at~\eqref{eq0.0} shows that the only term
of $F$ that contributes to items linear in $x_\alpha,y_\mu,z_p,w_a$ comes from
$8\sum_{a=1}^{m_1} q_a w_a$.
\end{remark}

We denote ${\bf q}^{*}$ by ${\bf q}^{*}(X,Y,Z)$, where
$X\in V_{+}^{*},Y\in V_{-}^{*}$ and $Z\in V_{0}^{*}$; thanks to Proposition~\ref{3.4-5}
we see that 
${\bf q}^{*}$ is a multilinear form in $X,Y,Z$.
We extend ${\bf q}^{*}(X,Y,Z)$ by requiring that ${\bf q}^{*}(e_0,Y,Z)=0$ and
${\bf q}^{*}(X,e_0,Z)=0$ for all $X,Y\in{\mathbb O}$. This is well-defined as
the right hand side of~\eqref{sub'} is $0$ if either $X=e_0$ or
$Y=e_0$. With this extension~\eqref{sub'} continues to hold.

\begin{lemma}\label{imp}
For $0\leq a,p\leq m_1$ and $X,Y\in{\mathbb O}$, we have
\begin{eqnarray}\label{identity1}
\aligned
&<{\bf q}^{*}(X,Y,e_a),e_a>=0,\\
&<{\bf q}^{*}(X,Y,e_0),X>=<{\bf q}^{*}(X,Y,e_0),Y>=0,\\
&<{\bf q}^{*}(e_a,Y,e_p),e_a>=-<{\bf q}^{*}(e_a{\overline e}_p,Y,e_0),e_a>,\\
&<{\bf q}^{*}(X,e_a,e_p),e_a>=-<{\bf q}^{*}(X,e_a\circ{\overline e}_p,e_0),e_a>.\\
&<{\bf q}^{*}(e_a,Y,e_a),e_p>=-<{\bf q}^{*}(e_p{\overline e}_a,Y,e_0),e_a>,\\
&<{\bf q}^{*}(X,e_a,e_a),e_p>=-<{\bf q}^{*}(X,e_p\circ{\overline e}_a,e_0),e_a>.
\endaligned
\end{eqnarray}
\end{lemma}

\begin{proof} Setting $i=a,j=b$ in~\eqref{defining'} and considering
the homogeneous part in $Y$ and $Z$ only, we obtain
\begin{eqnarray}\nonumber
\aligned
&\sum_{\alpha=0}^{m_1}<{\bf q}^{*}(e_\alpha,Y,Z),e_a><e_\alpha Z,e_b>\\
&+<{\bf q}^{*}(e_\alpha,Y,Z),e_b><e_\alpha Z,e_a>
=0.
\endaligned
\end{eqnarray}
Equivalently, it is
\begin{eqnarray}\label{identity4}
\aligned
&\sum_{\alpha=0}^{m_1}<{\bf q}^{*}(e_\alpha,Y,e_p),e_a><e_\alpha e_q,e_b>\\
&+\sum_{\alpha=0}^{m_1}<{\bf q}^{*}(e_\alpha,Y,e_q),e_a><e_\alpha e_p,e_b>\\
&+\sum_{\alpha=0}^{m_1}<{\bf q}^{*}(e_\alpha,Y,e_p),e_b><e_\alpha e_q,e_a>\\
&+\sum_{\alpha=0}^{m_1}<{\bf q}^{*}(e_\alpha,Y,e_q),e_b><e_\alpha e_p,e_a>=0.
\endaligned
\end{eqnarray}

Setting $q=a=b$ in~\eqref{identity4}, we see the first and the third sums
on the left are $0$, since they are simplified to $<q^{*}(e_0,Y,e_p),e_a>$.
Hence we obtain
$<{\bf q}^{*}(e_\alpha,Y,e_a),e_a>=0,$
where $e_\alpha$ is parallel to $e_a{\overline e}_p$ for any $p$.
Since $e_a{\overline e}_p$ runs through $e_0,\cdots,e_{m_1}$ when we vary $p$, we see
$<{\bf q}^{*}(e_\alpha,Y,e_a),e_a>=0$ for all $\alpha$. That is,
\begin{equation}\label{identity5}
<{\bf q}^{*}(X,Y,e_a),e_a>=0
\end{equation}
for all $X,Y,e_a$. In particular, the first identity of~\eqref{identity1}
is true.

On the other hand, setting
$a=b$ and $p=q=0$ we deduce the identity
$<{\bf q}^{*}(e_a,Y,e_0),e_a>=0$
for all $a$, which implies that

\begin{equation}\label{identity2}
<{\bf q}^{*}(X,Y,e_0),X>=0
\end{equation}
for all $X\in {\rm Im}({\mathbb O})$,
because any unit imaginary $X$ can serve as $e_a$, for some $a\neq 0,$ since the
group of automorphism of the normed algebra is transitive on the unit
imaginary sphere.
It follows from~\eqref{identity2},~\eqref{identity5}
for $a=0$,
and 
${\bf q}^{*}(e_0,Y,Z)=0$
that $<{\bf q}^{*}(X,Y,e_0),X>=0$ for all $X,Y\in{\mathbb O}$. Hence,
the second identity of~\eqref{identity1} is true.

The third identity of~\eqref{identity1} follows from setting $a=b$
and $q=0$.

The fifth identity comes from setting $p=b$ and $q=0$ and employing~\eqref{identity5}.

The fourth and sixth identities are derived from an equation similar to~\eqref{identity4} 
when, in~\eqref{defining'}, we look at the homogeneous part in $X$ and $Z$ only.
\end{proof}

\begin{corollary}\label{class} For $X,Y\in {\rm Im}({\mathbb O})$,
\begin{eqnarray}\nonumber
\aligned
&<{\bf q}^{*}(X,Y,Z),Z>=0,\quad Z\in{\rm Im}({\mathbb O})\;\;{\rm or}\;\;Z=e_0,\\
&<{\bf q}^{*}(X,Y,e_0),X>=<{\bf q}^{*}(X,Y,e_0),Y>=0,\\
&<{\bf q}^{*}(X,Y,Z),X>=-<{\bf q}^{*}(X{\overline Z},Y,e_0),X>,\quad Z\in{\mathbb O},\\
&<{\bf q}^{*}(X,Y,Z),Y>=-<{\bf q}^{*}(X,Y\circ {\overline Z},e_0),Y>,\quad Z\in{\mathbb O},\\
&<{\bf q}^{*}(X,Y,X),Z>=<{\bf q}^{*}(ZX,Y,e_0),X>,\quad Z\in{\mathbb O},\\
&<{\bf q}^{*}(X,Y,Y),Z>=<{\bf q}^{*}(X,Z\circ Y,e_0),Y>,\quad Z\in{\mathbb O}.
\endaligned
\end{eqnarray}
\end{corollary}

\begin{proof}
It follows from the identities, in order, of Lemma~\ref{imp} and the transitivity
of the automorphism group of ${\mathbb O}$ on its imaginary unit sphere.
\end{proof}

In fact, we can strengthen the first identity of Corollary~\ref{class} as follows.


\begin{lemma}\label{AI}
\begin{equation}\label{import}
<{\bf q}^{*}(U{\overline V},Y,V),W>=-<{\bf q}^{*}(W{\overline V},Y,V),U>,
\end{equation}
where $U,Y,W\in{\mathbb O}$ and $V$ is either $e_0$ or purely imaginary.
In particular, $<{\bf q}^{*}(X,Y,Z),W>$ is skew-symmetric for $Z$ and $W$
in ${\mathbb O}$. Moreover,
$<{\bf q}^{*}(X,Y,e_0),Z>$ is skew-symmetric in all $X,Y,Z\in{\mathbb O}$.
\end{lemma}

\begin{proof} Setting $p=q$ in~\eqref{identity4}, we obtain
$$
<{\bf q}^{*}(e_b{\overline e}_p,Y,e_p),e_a>=-<{\bf q}^{*}(e_a\overline{e}_p,Y,e_p),e_b>.
$$
The first statement follows.

Setting $U=e_0$ and $X:=W{\overline V}$ for a purely imaginary $V$, we obtain

\begin{eqnarray}\label{import'}
\aligned
&<{\bf q}^{*}(X,Y,V),e_0>=<{\bf q}^{*}(V,Y,V),XV>\\
&=-<{\bf q}^{*}(X,Y,e_0),V>,
\endaligned
\end{eqnarray}
where the last equality follows from the fifth identity of Corollary~\ref{class}.

The second statement is a consequence of~\eqref{import'} and the first
identity of Corollary~\ref{class},
which says that $<{\bf q}^{*}(X,Y,Z),W>$ is skew-symmetric in $Z$ and $W$ when
$Z$ and $W$ are purely imaginary.

The third statement follows from anti-symmetrizing the $X$ and $Y$ slots of
the two equations, respectively, of the second identity of Corollary~\ref{class}.
\end{proof}

\begin{corollary}\label{anti}
For $W\in{\mathbb O}$, we have
$$
<{\bf q}^{*}(X,Y,W),XW>=0\quad\text{and}\quad <{\bf q}^{*}(X,Y,W),Y\circ W>=0,
$$
so that anti-symmetrizing we get
\begin{eqnarray}\nonumber
\aligned
<{\bf q}^{*}(X,Y,U),XV>&=-<{\bf q}^{*}(X,Y,V),XU>\\
<{\bf q}^{*}(X,Y,U),Y\circ V>&=-<{\bf q}^{*}(X,Y,V),Y\circ U>
\endaligned
\end{eqnarray}
for $U,V\in{\mathbb O}$.
\end{corollary}

\begin{proof}
Setting $U=XW$ for $W\in\text{Im}({\mathbb O})$, we derive from~\eqref{import}
\begin{eqnarray}\nonumber
\aligned
&<{\bf q}^{*}(X,Y,W),XW>=<{\bf q}^{*}(U{\overline W},Y,W),U>\\
&=-<{\bf q}^{*}(U{\overline W},Y,W),U>=0.
\endaligned
\end{eqnarray}

We next calculate $<{\bf q}^{*}(X,Y,e_0),XW>$ for a purely imaginary $W$. By
the skew symmetry of $<{\bf q}^{*}(X,Y,e_0),Z>$ for all $X,Y,Z\in{\mathbb O}$,
\begin{eqnarray}\nonumber
\aligned
&<{\bf q}^{*}(X,Y,e_0),XW>=<{\bf q}^{*}(XW,Y,X),e_0>\\
&=-<{\bf q}^{*}(\overline{W}\,\overline{X},Y,X),e_0>
=<{\bf q}^{*}(e_0{\overline X},Y,X),{\overline W}>\\
&=-<{\bf q}^{*}(X,Y,X),\overline W>=<{\bf q}^{*}(X,Y,X),W>,
\endaligned
\end{eqnarray}
which cancels $<{\bf q}^{*}(X,Y,W),X>$ for an imaginary $W$. Putting all these
together, it follows that
\begin{equation}\label{goot}
<{\bf q}^{*}(X,Y,W),XW>=0
\end{equation}
for all $W\in{\mathbb O}$.

Likewise, $<{\bf q}^{*}(X,Y,W),Y\circ W>=0$ for all $W\in{\mathbb O}$
by a similar argument.
\end{proof}

\begin{remark} In fact, the first two identities of Corollary~{\rm \ref{anti}}
establish that
$<{\bf p}^{*},{\bf q}^{*}>=0$
by~\eqref{2fndtl}. This is the seventh of the ten equations of Ozeki and
Takeuchi~\cite[I, p 530]{OT}
defining an isoparametric hypersurface.
\end{remark}


We now come to a crucial observation. Recall the angle $\theta$ given before
Lemma~\ref{comparison}.

\begin{proposition}\label{CRUCIAL} Assume $\theta\neq 0$ and $\pi$. Let
$R(X,Y):={\bf q}^{*}(X,Y,e_0)$. Then
$$
R(X,Y)=XY-Y\circ X,
$$
if $e$ is perpendicular to $X,Y$ and $XY$, while
$$
R(X,Y)=\pm(XY-Y\circ X)
$$
if $XY$ is parallel to $e$.
\end{proposition}

\begin{proof} By Lemma~\ref{sub} we see $|R(Z,Z)|=|ZZ-Z\circ Z|=0$, so that $R(Z,W)$
is skew-symmetric in $Z$ and $W$. 

We may assume $X,Y\in{\rm Im}({\mathbb O})$ are orthonormal vectors
such that $X,Y$ and $XY$ are all perpendicular to $e$, where $e$ is given before
Lemma~\ref{comparison}.
Then $e_0,X,Y,XY,e,Xe,Ye,(XY)e$ form an octonian basis of
${\mathbb O}$.
It follows that
$R(X,Y)$ is a linear combination of the above basis elements. We know 
$$
<R(X,Y),e_0>=<R(X,Y),X>=<R(X,Y),Y>=0
$$
by the first two identities of Corollary~\ref{class}. Therefore,
we conclude          
\begin{equation}\label{inv}
R(X,Y)=a(XY)+fe+c(Xe)+d(Ye)+b((XY)e)
\end{equation}
for some functions $a,b,c,d,f$ on ${\mathcal M}$.

Let $X=g^{-1}(X'),Y=g^{-1}(Y')$ and $e=g^{-1}(e')$ for any automorphism $g$ of
${\mathbb O}$. Then
\begin{equation}\nonumber
\aligned
(g\cdot R)(X',Y')&:=g(R(g^{-1}(X'),g^{-1}(Y')))=g(R(X,Y))\\
&=a(X'Y')+fe'+c(X'e')+d(Y'e')+b((X'Y')e').
\endaligned
\end{equation}
The interpretation is that $(g\cdot R)(X',Y')$ is $R(X,Y)$ relative to
the new octonian basis $e_0,g^{-1}(e_1),\cdots,g^{-1}(e_7)$ with coordinates
$X',Y'$ and $e'$. Since any such $(X,Y,e)$ can be
$(g^{-1}(X'),g^{-1}(Y'),g^{-1}(e'))$ for a fixed 
$(X',Y',e')$ (think of it as $(e_1,e_2,e_4)$) as we vary $g$, we see that $a,b,c,d,f$ are all constant. But then
homogenizing $X$ and $Y$ in~\eqref{inv} shows that $c=d=0$ for (polynomial) degree reason,
and, moreover, that $f=0$ since R(X,Y) is skew-symmetric. 
So now
\begin{equation}\label{great}
R(X,Y)=a(XY)+b((XY)e).
\end{equation}
To determine $a$ and $b$, we note that by Lemma~\ref{AI} 
$$
<R(U,V),W>=<{\bf q}^{*}(U,V,e_0),W>
$$
is skew-symmetric in all variables. Hence the 3rd identity of Corollary~\ref{class}
gives
$$
<R(X,Y),XY>=<{\bf q}^{*}(X,Y,X),Y>,
$$
while the 4th identity of Corollary~\ref{class} gives
$$
<R(X,Y),Y\circ X>=-<{\bf q}^{*}(X,Y,X),Y>.
$$
Adding these two equations, incorporating Lemma~\ref{comparison} and bearing in mind
that $a=<R(X,Y),XY>$ and $b=<R(X,Y),(XY)e>$, we obtain
$$
a(1-\cos(2\theta))-b\sin(2\theta)=0.
$$
But then
$$
a^2+b^2=|R(X,Y)|^2=|XY-Y\circ X|^2=2+2\cos(2\theta)
$$
results in
$$
a=\pm(1+\cos(2\theta)),\quad b=\pm\sin(2\theta).
$$
(The signs for $a$ and $b$ agree.) By changing $e$ to $-e$, we may assume the
sign is positive. It follows that
$$
R(X,Y)=(1+\cos(2\theta))XY+\sin(2\theta)(XY)e=XY-Y\circ X.
$$

In the case when the orthonormal imaginary $X$ and $Y$ are such that
$XY=e$, we form an octonian basis $e_0,X,Y,e,W,WX,WY,We$. We have,
since $X\circ Y=XY=e$ by Lemma~\ref{comparison} and since
$R(X,Y)$ is skew-symmetric, that
$$
<R(X,Y),W>=<R(W,X),Y>=<WX-X\circ W,Y>=0
$$
by the previous case. In other words, $R(X,Y)$ is in the span of $e_0$ and $e$
since $<R(X,Y),X>=<R(X,Y),Y>=0$. Write
$$
<R(X,Y)=ae+be_0.
$$
Now, $b=<R(X,Y),e_0>=0$ by skew symmetry. Moreover,
since $|R(X,Y)|=|XY-Y\circ X|=2$, we see $a=\pm 2$ and
$$
R(X,Y)=\pm 2e=\pm 2XY=\pm(XY-Y\circ X).
$$
\end{proof}

\begin{corollary}\label{cute} $R(X,Y)=XY-YX$ if $\theta=0$ and $R(X,Y)=0$ if $\theta=\pi$.
\end{corollary}

\begin{proof} $e$ is arbitrary in~\eqref{great} when $\theta=0$ or $\pi$. Hence the real number $b=0$,
so that $R(X,Y)=aXY$. In the case when $\theta=\pi$ we have $a\circ b=ba$ for all
$a,b$ and $|R(X,Y)|=|XY-Y\circ X|=0$. So $a=0$. For $\theta=0$, i.e., when $a\circ b=ab$ for all $a,b$,
$|R(X,Y)|=2|X||Y|$.
So, $a=\pm 2$. Since changing $X,Y,Z$ to $-X,-Y,-Z$ leaves the 2nd fundamental form
fixed and changes the 3rd fundamental form by a sign, we may choose the positive sign.
\end{proof}

\section{Classification of ${\bf q}^{*}$}
We have seen in Lemma~\ref{sub} that the $3$rd fundamental form
${\bf q}^{*}$ satisfies
\begin{equation}\label{identity}
|{\bf q}^{*}(X,Y,Z)|=|X(Y\circ Z)-Y\circ(XZ)|.
\end{equation}
We now prove that
there are only three possibilities for ${\bf q}^{*}$.
\begin{theorem}\label{Th} Up to isometry, the possible ${\bf q}^{*}$ are either 
$$
{\bf q}^{*}(X,Y,Z)=(XY-YX)Z
$$
constructed by Ozeki and Takeuchi, where $\circ$ coincides with the octonian multiplication, or
$$
{\bf q}^{*}(X,Y,Z)=X(Y\circ Z)-Y\circ (XZ)
$$
constructed by Ferus, Karcher and M\"{u}nzner, where either $a\circ b=ab$ or
$a\circ b=ba$ for all $a,b\in{\mathbb O}$.
\end{theorem}

The proof of Theorem~\ref{Th} consists of a series of lemmas and corollaries
in the following subsections.


\subsection{The case when $\theta\neq 0$ and $\pi$}

\begin{lemma}\label{sub5} Suppose $\theta\neq 0$ and $\pi$. Let $X$ and $Y$ be purely imaginary and
perpendicular vectors in ${\mathbb O}$ and let $W$ be in the orthogonal
complement of the quaternion algebra
${\mathcal A}$ generated by $X$ and $Y$. Then
$$
{\bf q}^{*}(X,Y,W)=X(Y\circ W)-Y\circ (XW)
$$
if $e$ is perpendicular to ${\mathcal A}$, while
$$
{\bf q}^{*}(X,Y,W)=\pm(X(Y\circ W)-Y\circ (XW))
$$
if $XY$ is parallel to $e$; here, the sign agrees with that of $R(X,Y)$.
\end{lemma}

\begin{proof} We may assume $X,Y$ are unit vectors. Suppose $X,Y$ and $XY$
are all perpendicular to $e$.
Complete it to an octonian basis $e_0,X,Y,XY,e,Xe,Ye,(XY)e$
of ${\mathbb O}$. The third identity in Corollary~\ref{class}
and Proposition~\ref{CRUCIAL}
imply that
\begin{equation}\nonumber
\aligned
<{\bf q}^{*}(X,Y,e),X>&=<R(X,Y),Xe>=<XY-Y\circ X,Xe>\\
&=2\sin(2\theta)<(XY)e,Xe>=0.
\endaligned
\end{equation}
Likewise, the fourth identity in Corollary~\ref{class} and Proposition~\ref{CRUCIAL}
imply
$$
<{\bf q}^{*}(X,Y,e),Y>=<R(X,Y),Y\circ e>=<XY-Y\circ X,Y\circ e>=0.
$$
Meanwhile,
$$
<{\bf q}^{*}(X,Y,e),e_0>=-<{\bf q}^{*}(X,Y,e_0),e>=-<XY-Y\circ X,e>=0.
$$
On the other hand,
$$
<{\bf q}^{*}(X,Y,e),Xe>=<{\bf q}^{*}(X,Y,e),Ye>=0
$$
by the first two identities of Corollary~\ref{anti}. Lastly,
$<{\bf q}^{*}(X,Y,e),e>=0$ by the first identity of Corollary~\ref{class}.
In conclusion,
\begin{equation}\label{QQ}
{\bf q}^{*}(X,Y,e)=a(XY)+b((XY)e).
\end{equation}

To determine $a$ and $b$, setting $U=e$ and $V=Y$ in the 3rd equation in
Corollary~\ref{anti}, we deduce
\begin{eqnarray}\label{span}
\aligned
<{\bf q}^{*}(X,Y,e),XY>&=-<{\bf q}^{*}(X,Y,Y),Xe>\\
&=<{\bf q}^{*}(X,Y,e_0),(Xe)\circ Y>\\
&=<XY-Y\circ X,(Xe)\circ Y>=\sin(2\theta).
\endaligned
\end{eqnarray}
In the same vein,
\begin{eqnarray}\nonumber
\aligned
&<{\bf q}^{*}(X,Y,e),Y\circ X>=-<{\bf q}^{*}(X,Y,X),Y\circ e>\\
&=<{\bf q}^{*}(X,Y,e_0),(Y\circ e)X>=<XY-Y\circ X,(Y\circ e)X>\\
&=<XY-Y\circ X,(Ye)X>=\sin(2\theta),
\endaligned
\end{eqnarray}
while its left hand side simplifies to
\begin{eqnarray}\nonumber
\aligned
&<{\bf q}^{*}(X,Y,e),Y\circ X>
=<{\bf q}^{*}(X,Y,e),\cos(2\theta)YX+\sin(2\theta)(YX)e>\\
&=-\cos(2\theta)\sin(2\theta)-\sin(2\theta)<{\bf q}^{*}(X,Y,e),(XY)e>
\endaligned
\end{eqnarray}
by~\eqref{span}. So, when $\theta\neq\pi/2$, we end up with
$$
<{\bf q}^{*}(X,Y,e),(XY)e>=-(1+\cos(2\theta)),
$$
which is exactly
$$
{\bf q}^{*}(X,Y,e)=X(Y\circ e)-Y\circ(Xe).
$$
We then
use the third identity of Corollary~\ref{anti} to see that 
$$
{\bf q}^{*}(X,Y,W)=X(Y\circ W)-Y\circ(XW).
$$
for $W=Xe,Ye,(XY)e$, and hence for all $W$ perpendicular to ${\mathcal A}$.

When $\theta=\pi/2$, a straightforward calculation gives
\begin{equation}\label{ind}
|{\bf q}^{*}(X,Y,e)|=|X(Y\circ e)-Y\circ (Xe)|=1+\cos(2\theta)=0,
\end{equation}
so that once more
$$
{\bf q}^{*}(X,Y,e)=X(Y\circ e)-Y\circ (Xe)\; (=0).
$$

In the case when $XY=e$, we know $R(X,Y)=\pm(XY-YX)=\pm 2e$. We form an octonian basis
$e_0,X,Y,e,W,WX,WY,We$.
Then
\begin{eqnarray}\nonumber
\aligned
<{\bf q}^{*}(X,Y,W),e_0>&=-<R(X,Y),W>=<\pm 2e,W>=0,\\
<{\bf q}^{*}(X,Y,W),X>&=<R(X,Y),WX>=0,\\
<{\bf q}^{*}(X,Y,W),Y>&=<R(X,Y),W\circ Y>=0,\\
<{\bf q}^{*}(X,Y,W),W>&=0,\\
<{\bf q}^{*}(X,Y,W),XW>&=<{\bf q}^{*}(X,Y,W),YW>=0,
\endaligned
\end{eqnarray}
where the last identity follows from Corollary~\ref{anti}. It follows that
$$
{\bf q}^{*}(X,Y,W)=a(XY)+b(W(XY))
$$
for some $a,b\in{\mathbb R}$. But then for (polynomial) degree reason $a=0$.
Since 
$$
X(Y\circ W)-Y\circ (XW)=2\cos(2\theta)W(XY),
$$
we see by~\eqref{sub'} that
$$
{\bf q}^{*}(X,Y,W)=\pm(X(Y\circ W)-Y\circ(XW)).
$$


\end{proof}

\begin{corollary}\label{C} Suppose $\theta\neq 0$ and $\pi$. Let $X$ and $Y$ be purely imaginary and
perpendicular vectors in ${\mathbb O}$ and let $W$ be in the quaternion algebra
${\mathcal A}$ generated by $X$ and $Y$. Then
\begin{equation}\label{Hee}
{\bf q}^{*}(X,Y,W)=X(Y\circ W)-Y\circ (XW)
\end{equation}
if $e$ is perpendicular to ${\mathcal A}$, while
\begin{equation}\label{Haa}
{\bf q}^{*}(X,Y,W)=\pm(X(Y\circ W)-Y\circ (XW))
\end{equation}
if $XY$ is parallel to $e$; here, the sign agrees with that of $R(X,Y)$.
\end{corollary}

\begin{proof} The proof follows the same line of thoughts as in the preceding lemma.
Thus we shall only indicate the essential point.

We first assume that $e$ is perpendicular to ${\mathcal A}$ so that by the preceding lemma
\begin{equation}\label{following}
{\bf q}^{*}(X,Y,Z)=X(Y\circ Z)-Y\circ(XZ)
\end{equation}
for $Z$ perpendicular to ${\mathcal A}$. Then as before we construct an octonian basis
$e_0,X,Y,XY,e,Xe,Ye,(XY)e$. We know $<{\bf q}^{*}(X,Y,X),e_0>=-<R(X,Y),X>=0$
and $<{\bf q}^{*}(X,Y,X),X>=0$. By the 5th identity of Corollary~\ref{class},
\begin{eqnarray}\nonumber
\aligned
<{\bf q}^{*}(X,Y,X),Y>&=<R(X,Y),XY>\\
&=<XY-Y\circ X,XY>=1+\cos(2\theta).
\endaligned
\end{eqnarray}
For $Z$ perpendicular to ${\mathcal A}$, we use~\eqref{following} to see
$$
<{\bf q}^{*}(X,Y,X),Z>=-<{\bf q}^{*}(X,Y,Z),X>=<(Ze)(XY),X>,
$$
so that we derive
\begin{equation}\label{E}
<{\bf q}^{*}(X,Y,X),e>=<{\bf q}^{*}(X,Y,X),Xe>=<{\bf q}^{*}(X,Y,X),(XY)e>=0,
\end{equation}
while
\begin{equation}\label{EE}
<{\bf q}^{*}(X,Y,X),Ye>=-\sin(2\theta).
\end{equation}
Therefore, we conclude
\begin{equation}\label{Umm}
{\bf q}^{*}(X,Y,X)=(1+\cos(2\theta))Y-\sin(2\theta)Ye=X(Y\circ X)-Y\circ(XX).
\end{equation}
(Note that ${\bf q}^{*}=0$ if $\theta=\pi/2$.)
When $XY=e$, we from the octonian basis $e_0,X,Y,e,W,XW,YW,(XY)W$ and
we have $R(X,Y)=\pm 2XY$ and ${\bf q}^{*}(X,Y,Z)=\pm(X(Y\circ Z)-Y\circ(XZ))$
for $Z$ perpendicular to ${\mathcal A}$. We see $<{\bf q}^{*}(X,Y,X),Y>=\pm 2$
and $<{\bf q}^{*}(X,Y,X),Z>=0$ for all $Z$ perpendicular to ${\mathcal A}$. Hence
$$
{\bf q}^{*}(X,Y,X)=\pm 2Y=\pm (X(Y\circ X)-Y\circ(XX)).
$$
\end{proof}

\begin{theorem}\label{TM} Suppose $\theta\neq 0$ and $\pi$. For all $X,Y\in{\mathbb O}$ and all $Z\in{\mathbb O}$
we have
\begin{equation}\label{hot}
{\bf q}^{*}(X,Y,Z)=X(Y\circ Z)-Y\circ (XZ).
\end{equation}
Thus the hypersurfaces are of the type constructed by Ferus, Karcher and
M\"{u}nzner.
\end{theorem}

\begin{proof} Lemma~\ref{sub5} and Corollary~\ref{C} only deal with the case
when the imaginary $X$ and $Y$
are perpendicular in ${\bf q}^{*}(X,Y,Z)$, which leaves an undetermined sign. We now remove the sign
by considering the case when $X=Y$. 

Let $X,Y\in\text{Im}({\mathbb O})$ be orthonormal
such that $e$ is perpendicular to $X,Y$ and $XY$. Then the circles
$X(t):=\cos(t)X+\sin(t)Y$ and $Y(t):=-\sin(t)X+\cos(t)Y$ satisfy that $X(t), Y(t),X(t)Y(t)$
are perpendicular to $e$. Differentiating~\eqref{Hee} at $t=0$,
we obtain
\begin{equation}\nonumber
\aligned
&{\bf q}^{*}(Y,Y,W)-{\bf q}^{*}(X,X,W)\\
&=-(X(X\circ W)-X\circ(XW))+(Y(Y\circ W)-Y\circ(YW)).
\endaligned
\end{equation}
Note that
\begin{equation}\label{Good}
|{\bf q}^{*}(X,X,Z)|=|\sin(2\theta)(X((XZ)e)-(X(XZ))e)|\neq 0
\end{equation}
unless $\theta=\pi/2$. Homogenizing and comparing polynomial types, we get
$$
{\bf q}^{*}(X,X,W)=X(X\circ W)-X\circ(XW)
$$
when $\theta\neq\pi/2$. On the other hand, when $\theta\neq \pi/2$,
we fix the same $X$ and choose a $Y$ such that $XY=e$,
differentiating~\eqref{Haa} gives
$$
{\bf q}^{*}(X,X,W)=\pm((X(X\circ W)-X\circ(XW)).
$$
Therefore, the sign must be positive when $\theta\neq\pi/2$.

When $\theta=\pi/2$, the formula~\eqref{Good} implies ${\bf q}^{*}(X,X,Z)=0$
for all $X,Z\in{\mathbb O}$, and so ${\bf q}^{*}$ is skew-symmetric in $X$
and $Y$. So, a priori the sign is undetermined. However,
by~\eqref{ind} and~\eqref{Umm} we have seen ${\bf q}^{*}(X,Y,Z)=0$
for all $Z$ when $e$ is perpendicular to $X,Y$ and $XY$. The sign is
ambiguous only in the case when $XY=e$.
Now, set $e=e_4$. Then since
any two different imaginary basis elements $e_a,e_b\neq e_4$ satisfy either
$e_ae_b=e_4$, or $e_a,e_b$ and $e_ae_b$ are all perpendicular to $e_4$, the
analysis in Lemma~\ref{sub5} and Corollary~\ref{C} provides a recipe for
writing down ${\bf q}^{*}(X,Y,Z)$ explicitly as follows.
$$
{\bf q}^{*}(X,Y,Z)=\pm\sum(x_iy_je_i(\circ (e_j Z))-y_jx_ie_j\circ(e_iZ)),
$$
where $i,j\geq 1$ run over the indexes where $e_ie_je_4=\pm e_0$.

Since changing $X,Y,Z$ to $-X,-Y,-Z$ retains the 2nd fundamental form and
changes the 3rd fundamental form by a sign, we might as well choose the positive sign.

Therefore, in any event, the 3rd fundamental form is the desired form given
by~\eqref{hot}. 
\end{proof}

Proposition~\ref{deform} implies that we can always perturb to find a mirror point
$x^{*}\in M_{-}$ at which $\theta=0$ or $\pi$, even when initailly the choice of $x^{*}$
produces an angle $\theta$ different from $0$ and $\pi$. Therefore, the classification is reduced to the
case when $\theta=0$ or $\pi$.

\subsection{The case when $\theta=0$ or $\pi$}
By Corollary~\ref{cute}, we know $R(X,Y)=XY-YX$ for $\theta=0$ and $R(X,Y)\equiv 0$ for $\theta=\pi$.

\begin{corollary}\label{class1} Suppose $a\circ b=ab,\forall a,b.$ For $X,Y\in {\rm Im}({\mathbb O})$, we have
\begin{eqnarray}\nonumber
\aligned
&<{\bf q}^{*}(X,Y,Z),Z>=0,\\
&<{\bf q}^{*}(X,Y,e_0),X>=<{\bf q}^{*}(X,Y,e_0),Y>=0,\\
&<{\bf q}^{*}(X,Y,Z),X>=2<X,Y><X,Z>-2|X|^2<Y,Z>,\\
&<{\bf q}^{*}(X,Y,Z),Y>=-2<X,Y><Y,Z>+2|Y|^2<X,Z>.
\endaligned
\end{eqnarray}
\end{corollary}

\begin{proof} This follows from $R(X,Y)=XY-YX$ and Corollary~\ref{class}.
\end{proof}

\begin{corollary}\label{quaternion} If the normed algebra
is ${\mathbb H}$, then Theorem~\ref{Th} is true.
\end{corollary}

\begin{proof} By Remark~\ref{Q}, either $a\circ b=ab$ or $=ba$ for all $a,b\in{\mathbb H}$.

Case 1. $a\circ b=ba,\forall a,b.$

Then by~\eqref{identity}, $|{\bf q}^{*}(X,Y,Z)|=|X(ZY)-(XZ)Y|=0$
by the associativity of ${\mathbb H}$. So,
$$
{\bf q}^{*}=0=X(Y\circ Z)-Y\circ(XZ).
$$
The hypersurface is of the type constructed by Ferus, Karcher and M\"{u}nzner
by Section~\ref{sec5}.

\noindent Case 2. $a\circ b=ab,\forall a,b.$

Let $X,Y$ be
mutually orthogonal and purely imaginary. We set $Z=XY$. Then the first, third
and fourth identities of
Corollary~\ref{class1} imply ${\bf q}^{*}(X,Y,Z)$ is perpendicular to $X,Y,Z$;
therefore,
${\bf q}^{*}(X,Y,Z)$ is parallel to $e_0$.
Let ${\bf q}^{*}(X,Y,Z)=-2c|X|^2|Y|^2e_0$ for some constant $c$.
By identity~\eqref{identity} we obtain the identity $|{\bf q}^{*}(X,Y,Z)|=2|X|^2|Y|^2$;
we see therefore $c=\pm 1$. Thus,
$$
{\bf q}^{*}(X,Y,Z)=-2c|X|^2|Y|^2e_0=2cZZ=c(XY-YX)Z.
$$
Meanwhile,
$$
{\bf q}^{*}(X,Y,e_0)=R(X,Y)=(XY-YX)e_0.
$$
Corollary~\ref{class1} also yields

\begin{eqnarray}\nonumber
\aligned
&{\bf q}^{*}(X,Y,X)=2|X|^2=(XY-YX)X,\\
&{\bf q}^{*}(X,Y,Y)=-2|Y|^2X=(XY-YX)Y.
\endaligned
\end{eqnarray}
Putting all these together, we arrive at
\begin{eqnarray}\label{double}
\aligned
{\bf q}^{*}(X,Y,W)=&(XY-YX)W,\quad\text{or}\\
{\bf q}^{*}(X,Y,W)=&(XY-YX)W-<W,XY-YX>e_0,
\endaligned
\end{eqnarray}
where $c=1$ for the first equation and $c=-1$ for the second. Although we have derived the formulae assuming that $X$ and $Y$ are
perpendicular, the same formulae remain true for any two imaginary $X$ and $Y$
since ${\bf q}^{*}(U,V,W)$ is skew-symmetric in $U,V$.

If $c=1$, then
$$
{\bf q}^{*}(X,Y,W)=X(Y\circ W)-Y\circ(XW).
$$
So the hypersurface is of the type constructed by Ferus, Karcher and M\"{u}nzner
by Section~\ref{sec5}. It satisfies~\eqref{goot}
\begin{equation}\label{bridging}
<X(Y\circ W)-Y\circ(XW),XW>=0
\end{equation}
We show $c=-1$ is impossible. Assume otherwise.
Then since such an isoparametric hypersurface must also satisfy~\eqref{goot}, we
would conclude
\begin{equation}\nonumber
\aligned
0&=<{\bf q}^{*}(X,Y,W),XW>\\
&=<X(Y\circ W)-Y\circ(XW),XW>-<W,XY-YX>e_0,XW>\\
&=<W,XY-YX><W,X>\neq 0
\endaligned
\end{equation}
by~\eqref{bridging}.
This is a contradiction.
\end{proof}


To finish Theorem~\ref{Th} in the octonian case, we break it into two cases.

\noindent Case 1. $a\circ b=ab,\forall a,b.$

Identity~\eqref{identity} shows that $|{\bf q}^{*}(X,X,Z)|=0,\forall X,Z\in{\mathbb O},$
so that ${\bf q}^{*}(X,Y,Z)$ is skew-symmetric in $X,Y,\forall X,Y\in{\mathbb O}.$

Let $X,Y\neq 0$ be perpendicular and purely imaginary and $W$ be
in the orthogonal complement of ${\mathcal A}$, the quaternion algebra generated
by $X$ and $Y$. We know by~\eqref{QQ} and~\eqref{span} 
that ${\bf q}^{*}(X,Y,W)=\pm 2((XY)W)$, if $X,Y$ and $XY$
are all perpendicular to $e$, and the same formula holds if $XY=e$, where the signs
might not be related a priori in the two cases. We assume first that the signs
are identical. Namely,
$$
{\bf q}^{*}(X,Y,W)=2c((XY)W),
$$
where $c=1$ or $c=-1$ for all $W$ perpendicular to ${\mathcal A}$. 
If $c=1$, then
$$
{\bf q}^{*}(X,Y,W)=(XY-YX)W,
$$
which remains true for any
two purely imaginary $X$ and $Y$ not necessarily perpendicular to each other,
as ${\bf q}^{*}$ is skew-symmetric in $X,Y$. It follows that
$$
{\bf q}^{*}(X,Y,Z)=(XY-YX)Z
$$
for any $Z\in{\mathbb O}$, as it is also true for $Z\in {\mathcal A}$ by
Corollary~\ref{quaternion},
where we use~\eqref{E} and~\eqref{EE} to see
that ${\bf q}^{*}(X,Y,Z)\in{\mathcal A}$ for $Z\in{\mathcal A}$.
This is the isoparametric hypersurface constructed by
Ozeki  and Takeuchi.

If $c=-1$, then
$$
{\bf q}^{*}(X,Y,W)=-2(XY)W=X(YW)-Y(XW),
$$
so that there holds
$$
{\bf q}^{*}(X,Y,Z)=X(YZ)-Y(XZ)=X(Y\circ Z)-Y\circ (XZ)
$$
for any $X,Y,Z\in{\mathbb O}$, as it is true for $Z\in {\mathcal A}$ by
Corollary~\ref{quaternion}.
These are the isoparametric hypersurfaces
constructed by Ferus, Karcher and M\"{u}nzner.

We need to remove the case when 
${\bf q}^{*}(X,Y,W)=2((XY)W)$ if $X,Y,$ and $XY$
are all perpendicular to $e$, whereas
${\bf q}^{*}(X,Y,W)=-2((XY)W)$ when $XY=e$. Assuming this is the case. Then 
Corollary~\ref{quaternion} implies
$$
{\bf q}^{*}(X,Y,W)=(XY-YX)W+h(X,Y,W),
$$
where $h(X,Y,W)=-4eW^{\perp}$ if $XY=e$. As seen in Corollary~\ref{quaternion},
the existence of an isoparametric hypersurface with such a ${\bf q}^{*}$
would imply
$$
<h(X,Y,W),XW>=<{\bf q}^{*}(X,Y,W)-(XY-YX)W,XW>=0.
$$
But then if we pick $XY=e$ and $W=e_0+Xe$, we get
$$
<h(X,Y,W),XW>=-4<eW^{\perp},XW>=4<X,X-e>\neq 0.
$$
This is a contradiction.

\noindent Case 2. $a\circ b=ba,\forall a,b.$

Note that again $|{\bf q}^{*}(U,U,Z)|=|U(ZU)-(UZ)U|=0,\forall U,Z\in{\mathbb O},$ so that
${\bf q}^{*}$ is skew-symmetric in the first two slots.

If $c=1$, then
$$
{\bf q}^{*}(X,Y,W)=2(XY)W=X(WY)-(XW)Y,
$$
so that
$$
{\bf q}^{*}(X,Y,Z)=X(ZY)-(XZ)Y=X(Y\circ Z)-Y\circ(XZ)
$$
for any $X,Y,Z\in{\mathbb O}$, as ${\bf q}^{*}=0$ on ${\mathcal A}$.

If $c=-1$, then ${\bf q}^{*}$ only differs from the previous case by a negative sign.
Changing $X,Y,Z$ to $-X,-Y,-Z$ converts it to the previous case.

This completes the classification of Theorem~\ref{Th}.

\begin{remark}
In the octonian case, the two
isoparametric hypersurfaces
with ${\bf q}^{*}=X(Y\circ Z)-Y\circ (XZ)$ constructed by Ferus, Karcher and
M\"{u}nzner
are of Condition B at $x^{*}\in M_{-}$. In contrast, the hypersurface with ${\bf q}^{*}=(XY-YX)Z$
is not of Condition B at $x^{*}$; however, it is of both Condition A and B
at $x\in M_{+}$ constructed by Ozeki and Takeuchi.

In the quaternionic case, however, $(XY-YX)Z=X(YZ)-Y(XZ)$, so that we have
only two different
such isoparametric hypersurfaces, where the example of Ozeki and Takeuchi of
multiplicities
$(3,4)$ of Conditions A and B at $x\in M_{+}$ is also of Condition B at
$x^{*}\in M_{-}$.
The other isoparametric hypersurface is of Condition B at $x^{*}\in M_{-}$
with
$q^{*}=X(ZY)-(XZ)Y=0$; it is the homogeneous example of multiplicities
$(4,3)$.
\end{remark}

\end{document}